\documentclass{amsart}
\usepackage{amsmath}
\usepackage{amssymb}
\usepackage[english]{babel}
\usepackage{graphics}
\usepackage[all]{xy}
\usepackage{amsrefs}
\usepackage{amsmath}
\usepackage{amsthm}
\usepackage{amssymb}
\usepackage{amsfonts}
\usepackage{amsxtra}     
\usepackage{tikz}
\usepackage{pgfplots}
\usepackage{multicol}
\usepackage{subfigure}

\usepackage[colorlinks=true,linkcolor=blue,urlcolor=blue, citecolor=blue]{hyperref}
\usepackage{fullpage}
\usepackage{epsfig}
\usepackage{verbatim}
\usepackage{enumitem}
\usepackage{algorithmicx}
\usepackage{float}
\usepackage{algpseudocode}

\newcommand{\newword}[1]{\textbf{\emph{#1}}} %use when you define a new word

\newtheorem{Theorem}{Theorem}[section]
\newtheorem{Proposition}[Theorem]{Proposition}
\newtheorem{Corollary}[Theorem]{Corollary}

\newtheorem{Conjecture}[Theorem]{Conjecture}

\theoremstyle{definition}
\newtheorem{Definition}[Theorem]{Definition}

\newtheorem{Example}[Theorem]{Example}

\newtheorem{Problem}[Theorem]{Problem}

\theoremstyle{remark}
\newtheorem{Remark}[Theorem]{Remark}

\let\bb\mathbb
\newcommand{\abs}[1]{\left\lvert #1 \right\lvert}

\usepackage{graphicx}
\usepackage{adjustbox}

\title{On Good Infinite Families of Toric Codes or the Lack Thereof}

\author[M.~Dolorfino]{Mallory~Dolorfino}
\address{Mallory~Dolorfino\\ Kalamazoo College \\ Kalamazoo, Michigan, USA \\
  \href{mailto:mallory.dolorfino19@kzoo.edu}
  {{\ttfamily\upshape mallory.dolorfino19@kzoo.edu}}}

\author[C.~Horch]{Cordelia~Horch}
\address{Cordelia Horch\\ Occidental College \\ Los Angeles, California\\ USA\\
 \href{mailto:chorch@oxy.edu}
  {{\ttfamily\upshape chorch@oxy.edu}}}

\author[K.~Jabbusch]{Kelly Jabbusch}
\address{Kelly Jabbusch\\ Department of Mathematics \& Statistics\\ University
  of Michigan--Dearborn \\ Dearborn, Michigan
 \\ USA \\ 
  \href{mailto:jabbusch@umich.edu}%
  {{\ttfamily\upshape jabbusch@umich.edu}}}
  
  \author[R.~Martinez]{Ryan~Martinez}
\address{Ryan~Martinez\\ Harvey Mudd College\\ Claremont\\ California\\ USA\\ 
  \href{mailto:rmmartinez@hmc.edu}
  {{\ttfamily\upshape rmmartinez@hmc.edu}}}

\begin{document}

\maketitle

\begin{abstract}Toric codes, introduced by Hansen, are the natural extensions of Reed-Solomon codes.  A toric code is a $k$-dimensional subspace of $\bb{F}_q^n$,  determined by a toric variety or its associated integral convex polytope $P \subseteq [0,q-2]^n$, where $k=|P\cap \mathbb{Z}^n|$ (the number of integer lattice points of $P$). 
There are two relevant parameters that determine the quality of a code: the information rate, which measures how much information is contained in a single bit of each codeword; and the relative minimum distance, which measures how many errors can be corrected relative to how many symbols each codeword has. 
Soprunov and Soprunova defined a good infinite family of codes to be a sequence of codes of unbounded polytope dimension such that neither the corresponding information rates nor relative minimum distances go to 0 in the limit. 
We examine different ways of constructing families of codes by considering polytope operations such as the join and direct sum. In doing so, we give conditions under which no good family can exist and strong evidence that there is no such good family of codes.
\end{abstract}

\section{Introduction}
%Brief history

Toric codes are the natural extensions of  Reed-Solomon codes.  Fix a finite field $\bb{F}_q$ and let $P \subseteq \mathbb{R}^n$ be an integral convex polytope, which is contained in the $n$-dimensional box $[0,q-2]^n$ and has $k$ lattice points.  The toric code $C_P(\bb{F}_q)$ is obtained by evaluating linear combinations of monomials corresponding to the lattice points of $P$ over the finite field $\bb{F}_q$. Toric codes were introduced by Hansen for $n=2$, (see \cite{Hansen98} and \cite{Hansen}), and subsequently studied by various authors including Joyner \cite{Joyner04},  Little and Schenck \cite{LSchenk}, Little and Schwarz \cite{LSchwarz}, Ruano \cite{ruano}, and Soprunov and Soprunova \cite{SS1}, \cite{SS2}. 

To determine if a given code is ``good" from a coding theoretic perspective, one considers parameters associated to the code:  the block length, the dimension, and the minimum distance (see Section \ref{preliminaries} for the precise definitions).  For the class of toric codes, the first two parameters are easy to compute.  Given a toric code $C_P(\bb{F}_q)$, the block length is $N=(q-1)^n$, and the dimension of the code is given by the number of lattice points $k= |P \cap \bb{Z}^n|$ \cite{ruano}. The minimum distance is not as easy to compute, but various authors have given formulas for computing (or bounding) minimum distances for toric codes coming from special classes of polytopes.  In the case of toric surface codes (where $n=2$), Hansen computed the minimum distance for codes coming from Hirzebruch surfaces (\cite{Hansen98}, \cite{Hansen}), Little and Schenck determined upper and lower bounds for the minimum distance of a toric surface code by examining Minkowski sum decompositions \cite{LSchenk}, and Soprunov and Soprunova improved these bounds for surface codes by examining the Minkowski length \cite{SS1}.  For toric codes arising from polytopes $P \subseteq \mathbb{R}^n$ with $n>2$, Little and Schwarz used Vandermonde matrices to compute minimum distances of codes from simplices and rectangular polytopes \cite{LSchwarz}, and Soprunov and Soprunova computed the minimum distance of a code associated to a product of polytopes and a code associated to a $k$-dialate of a pyramid over a polytope \cite{SS2}.   In general, an ideal code will have minimum distance and dimension large with respect to its block length.  Restating these parameters with this in mind, we will consider the relative minimum distance, $\delta(C_P)$, of a toric code $C_P$, which is the ratio of the minimum distance, $d(C_P)$, to the block length:  $\delta(C_P) = \frac{d(C_P)}{N}$, and the information rate, $R(C_P),$ which is the ratio of the dimension of the code to the block length: $R(C_P)= \frac{k}{N}.$

In this paper we investigate infinite families of toric codes, motivated by the work of Soprunov and Soprunova \cite{SS2}.  They define a \newword{good} infinite family of toric codes as a sequence of toric codes where neither the relative minimum distances nor information rates of the codes go to zero in the limit. %We define this formally in Definition \ref{infinte family def}.
\iffalse
Given a sequence $\{P_i\}$ of non-empty integral convex polytopes, such that $P_i \subseteq [0,q-2]^{n_i}$, we consider the associated toric codes $\{C_{P_i}\}$.  
We say $\{C_{P_i}\}$ is an infinite family of toric codes if for some $\delta, R \in [0,1]$ we have 
    $n_i \to \infty$, $\delta(P_i) \to \delta$, and $R(P_i) \to R$ as $i \to \infty$.  Furthermore the infinite family of toric codes is \newword{good} when neither $\delta$ nor $R$ is $0.$  In \cite{SS2} various infinite families of toric codes are constructed using polytope operations such as the Cartesian product and pyramids over polytopes.  However, in the examples given, none of the families constructed are good, and the authors propose the problem to find a good infinite family of codes.
\fi
     In Section \ref{infinite families}, we define this formally and explore methods of constructing infinite families through different polytope operations, in particular the join and the direct sum.  Using these operations, we construct examples of infinite families of toric codes, but like the examples constructed in \cite{SS2}, none are good. 
     
    These examples lead us to conjecture that there are in fact no examples of such good infinite families of  toric  codes.  In Section \ref{no good family} we build intuition and give evidence towards Conjecture \ref{conjecture}. In particular, we verify the conjecture for infinite families whose corresponding polytopes contain hypercubes of unbounded dimension or  have certain Minkowski lengths.  
    
    \begin{Theorem} Let $\{P_i\}$ be an infinite family of codes defined over the field $\bb{F}_q$.  
    \begin{enumerate}
        \item Proposition \ref{specialcase1}: If the $\{P_i\}$ contain hypercubes of unbounded dimension, then $\delta(P_i) \to 0$ and the infinite family is not good.
        \item Proposition \ref{specialcase2}: If the Minkowski  lengths of the polytopes $P_i$ are unbounded as $i \to \infty$, then $\delta(P_i) \to 0$ and the infinite family is not good. 
        \end{enumerate}
    \end{Theorem}
   
   We conjecture that if the Minkowski length is bounded then the information rate, $R(P_i)$, tends to $0$, and discuss some special cases.  
     We conclude the paper with some questions for future research if our conjecture is proved.  

\section{Preliminaries}\label{preliminaries}
%Somewhere in here add the L$S result about polytope transformations -- doesn't affect code
To construct a toric code over a finite field $\bb{F}_q$ for some prime power $q$, let $P \subseteq [0,q-2]^n \subseteq \bb{R}^n$ be an integral convex polytope with $k$ lattice points. 
%Note that in general $n$ is the smallest dimension that the polytope is in. 
The toric code $C_P(\bb{F}_q)$ is obtained by evaluating linear combinations of monomials corresponding to the lattice points of $P$ over the finite field $\bb{F}_q$. More concretely, $C_P(\bb{F}_q)$ is given by a generator matrix with each row corresponding to a lattice point $p = (p_1, p_2, \ldots, p_n)  \in (P \cap \bb{Z}^n)$ and each column corresponding to an element $a = (a_1, a_2, \ldots, a_n) \in (\bb{F}_q^\times){^n}$ defined by:
\begin{equation*}
    G = (a^p),
\end{equation*}
where $a^p=a_1^{p_1}a_2^{p_2} \cdots a_n^{p_n}$.
Equivalently, $C_P(\bb{F}_q)$ is the image of the following evaluation map. Let 
\begin{equation*}
    \mathcal{L}_P = \text{span}_{\bb{F}_q}\{\textbf{x}^p|\,p\in P\cap \bb{Z}^n\}
\end{equation*}
where $\textbf{x}^p = x_1^{p_1}\cdots x_n^{p_n}.$
 
If we choose an ordering of the elements of 
$(\mathbb{F}_q^\times)^n$, 
then the map
\begin{align*}
     \varepsilon: \mathcal{L}_P &\rightarrow \bb{F}_q^{{(q-1)}^n}\\
         f & \mapsto (f(a)|a \in (\bb{F}_q^\times){^n})
\end{align*}
evaluates polynomials $f$ in $\mathcal{L}_P$ at each of
the points of $(\bb{F}_q^\times){^n}$ to give a vector of
$\bb F_q^{(q-1)^n}$. The image of this map over all 
polynomials in $\mathcal{L}_P$ gives the toric code. It can be verified that $\varepsilon$ is a linear
map between the vector spaces $\mathcal L_P$ and $\bb F_q^{(q-1)^n}$. Thus the toric code is a vector subspace of $\bb F_q^{(q-1)^n},$ and is therefore a \newword{linear code}. If we choose $\{\textbf x^p | p \in P \cap \bb Z^n\}$ as our basis for $\mathcal L_P$, and the standard
basis of $\bb F_q^{(q-1)^n}$, then the matrix of $\varepsilon$ 
corresponds exactly with the generator matrix $G$. 
%Note that $\varepsilon$ is injective when $P$ lies in the $n$-box $[0,q-2]^n$. 
For simplicity we will generally omit the reference to $\bb F_q$ and refer to the toric code as $C_P$.

\begin{Example}\label{ex:UnitBox}
Let $q=5$ and $n = 2$. Consider the polytope $P \subseteq \bb{R}^2$ with the $k = 4$ lattice points $(0,0), (1,0), (0,1), (1,1)$ shown below.

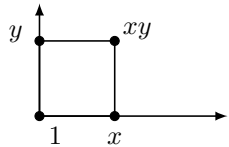
\begin{figure}[h!] \label{fig:box}
    \begin{center}
    \begin{tikzpicture}[scale=1]
    % axes
    \draw [semithick, ->, >=latex] (0,0) -- (2.5,0);
    \draw [semithick, ->, >=latex] (0,0) -- (0,1.5);
    % axes labels
    \node [left] at (-.1,1.1) {$y$};
    \node [below right] at (0,-0.02) {$1$};
    \node [below] at (1,-0.09) {$x$};
    \node [above] at (1.3,0.9) {$xy$};
    % square
    \draw [semithick, black] (0,0) -- (0,1) -- (1,1) -- (1,0) -- cycle;
    % lattice points
    \draw [fill] (1,1) circle [radius=0.06];
    \draw [fill] (0,0) circle [radius=0.06];
    \draw [fill] (1,0) circle [radius=0.06];
    \draw [fill] (0,1) circle [radius=0.06];
    \end{tikzpicture}
    \end{center}
    \caption{The polytope $P$.} 
    \end{figure}

The toric code $C_P(\bb{F}_5)$ is given by the $4 \times 16$ generator matrix $G$ below. In particular, each row of $G$ corresponds to a basis element of $\mathcal L_P$, $\{1, x, y, xy\}$ (corresponding to each of the four lattice points $p\in P$) and each column corresponds to a particular evaluation of the 
polynomial, one of the 16 elements $a \in (\bb{F}_5^\times)^2$. 
%The first row of $G$ corresponds to the lattice point $(0,0)$.
Finally, the elements of the toric code $C_P(\bb F_5)$ are obtained
by choosing the coefficients from $\bb F_5$ to use with the basis of $\mathcal L_P$. In particular, $$C_P(\bb F_5) = \{a G | a \in \bb F_5^4\}.$$
 The full generator matrix for our example is
\begin{equation*}
    G = 
    \begin{pmatrix}
1 & 1 & 1 & 1 & 1 & 1 & 1 & 1 & 1 & 1 & 1 & 1 & 1 & 1 & 1 & 1\\
1 & 1 & 1 & 1 & 2 & 2 & 2 & 2 & 3 & 3 & 3 & 3 & 4 & 4 & 4 & 4\\
1 & 2 & 3 & 4 & 1 & 2 & 3 & 4 & 1 & 2 & 3 & 4 & 1 & 2 & 3 & 4\\
1 & 2 & 3 & 4 & 2 & 4 & 1 & 3 & 3 & 1 & 4 & 2 & 4 & 3 & 2 & 1 
\end{pmatrix}
\end{equation*}
We can see that the first row corresponds to the basis element $1$ (the second to $x$ and so on),
and each column represents evaluating the row's basis element at a different member of $(\bb F_5^\times)^2$. We immediately see from this generator matrix that $C_P(\bb F_5)$ is an at most 
4 dimensional vector subspace of $\bb F_5^{16}$. It turns out that 
$C_P(\bb F_5)$ is exactly 4 dimensional since the polytope $P$ has exactly 4 lattice points.

\end{Example}

In general the quality of a toric code $C = C_P(\bb F_q)$, where $P$ is a convex integral lattice polytope living in $[0,q-2]^n$ can be classified by the following parameters:

\begin{Definition}
The \newword{block length}, $N$, is the number of entries in each of
$C$'s codewords. Based on the construction of toric codes, the block length of $C$ is exactly $(q-1)^n$.
\end{Definition}
Intuitively, the block length is the number of characters or symbols of $\bb F_q$ we \emph{write} 
to represent an element of our code $C$.
In the above example,  Example \ref{ex:UnitBox}, we have block length $(5-1)^2 = 16.$

\begin{Definition}
The \newword{dimension}, $k$, is $C$'s dimension over $\bb F_q$. If
$P$ has $\abs{P \cap \bb Z^n}$ lattice points, then the dimension of $C$ is exactly $k = \abs{P \cap \bb Z^n}$ (the proof of this equality is given in \cite{ruano}). 
\end{Definition}
Intuitively, the dimension is the number of characters or symbols of $\bb F_q$ we \emph{need} to
represent an element of our code $C$.
In the above example, Example \ref{ex:UnitBox}, the  dimension of $C_P$ is 4. That is, even though we write 16 symbols of $\bb F_5$ to
represent each codeword, we are only storing 4 symbols of $\bb F_5$ worth in data.

The reason we choose to use this storage inefficiency is so that we can correct errors in
our codewords. For example, suppose we are given the following codeword with an error 
from our running example above:
$$w = (1,1,1,1, ~ 1,1,1,1, ~ 1,1,1,1, ~ 1,1,1,0).$$
If we operate under the assumption that most errors occur independently on individual symbols of the code
with relatively small probability, then 
we can correct $w$ to 
$$w' = (1,1,1,1, ~ 1,1,1,1, ~ 1,1,1,1, ~ 1,1,1,1) = (1,0,0,0) ~ G,$$
which is indeed in our code $C_P(\bb F_5)$. This notion is formalized by Hamming distance: the
number of differences between two codewords. We see that the Hamming distance between $w$ and 
$w'$ is 1, and in fact there is no other $c \in C_P(\bb F_5)$ that is this close 
to $w$. As another example, consider
$$v = (0,0,0,0, ~ 1,1,1,1, ~ 0,0,0,0, ~ 1,1,1,1).$$
We see that both 
$$v' = (0,0,0,0, ~ 1,1,1,1, ~ 2,2,2,2, ~ 3,3,3,3) = (4,1,0,0) G$$ 
and
$$v'' = (0,0,0,0, ~ 0,0,0,0, ~ 0,0,0,0, ~ 0,0,0,0) = (0,0,0,0) G$$ 
are Hamming distance $8$ away from $v$, so we don't know what to correct to. It is customary
to use the minimum distance to measure what the largest error we can correct is.
\begin{Definition}
The \newword{minimum distance} $d$ of a code $C$ is the minimum Hamming distance
between any two codewords in $C$.  Equivalently, the minimum distance is the minimum weight of a codeword in $C$, where the weight of a codeword $w$ is the Hamming distance between $w$ and $(0, 0, \ldots, 0) \in \mathbb{F}_q^N.$
Because of this equivalence, we have  
for a toric code $C_P$ that the minimum distance is given by
\begin{equation*}
    d(C_P) = (q-1)^n - \max_{0 \neq f \in \mathcal{L}_P} \abs{Z(f)}.
\end{equation*}
where $Z(f)$ is the set of $\bb F_q$ zeros of $f$.  We'll use $N(P)$ to denote the $\max_{0 \neq f \in \mathcal{L}_P} \abs{Z(f)}$.
\end{Definition}

\begin{Remark}
The equivalence of the minimum Hamming distance between two codewords in $C$ and the minimum weight of a codeword in $C$ holds only for linear codes, and  is thus valid for our considerations.
\end{Remark}

Suppose we have a code $C$ with minimum distance $d$, and we are given an erroneous codeword $w$ and
valid codeword $w' \in C$ such that the Hamming distance between $w$ and $w'$ is $\ell \leq \lfloor (d-1)/2 \rfloor $. Since the Hamming distance defines a metric on $\bb F_q^N$, we know that for all other 
codewords $c \in C$ we have 
$$dist(c, w) \geq \abs{dist(c, w') - dist(w', w)} \geq d - \lfloor (d-1)/2 \rfloor 
= \lceil (d+1)/2\rceil > \lfloor (d-1)/2 \rfloor.$$
In other words, \emph{we know} that $w'$ is the closest codeword to $w$ in $C$. Thus, for a code
$C$ with minimum distance $d$ we know we can correct $\lfloor (d-1)/2 \rfloor$ errors.
%such as \newword{length} of a codeword, $(q-1)^n$, the  \newword{dimension} of the code, 
%measuring how many codewords is in the code, also is the number of lattice points for codes arising from polytopes since the map is one-to-one.
%which is the number of lattice points $k = |P \cap \bb{Z}^n|$ and \newword{minimum distance}, the minimum Hamming distance between two codewords. If we let $Z(f)$ denote the number of points in $(\bb{F}_q^\times){^n}$ where the polynomial $f \in \mathcal{L}_P$ vanishes, also called the zeros of $f$, then the minimum distance $d(C_P)$ is computed as follows:

Returning to Example \ref{ex:UnitBox}, to compute the minimum distance we must count the number of $(\bb{F}_5^{\times})^2$ zeros of polynomials in $\mathcal L_P = \text{span}\{1, x, y, xy\}$.
If $f(x,y) \in \mathcal L_P$ factors into two linear terms, say $f(x,y) = (x-a)(y-b)$, then if $x = a$, there are $(q-1)$ choices for $y$ and if $y=b, x\neq a$, then there are $(q-2)$ choices for $x$ such that $f$ evaluates to $0.$ Therefore $N(P) \geq (q-1) + (q-2) = 2q-3$, which, when $q=5$, is $7$. If $f(x,y) \in \mathcal L_P$ is an irreducible polynomial over $\bb{F}_5$, then the zero set of $f(x,y)$ is an affine conic section over $\bb{F}_5$, and hence contains at most six $\bb{F}_5$-rational points, by the Hasse-Weil bound.  Thus an irreducible polynomial can have at most six zeros.  Therefore, the maximum number of zeros of $f(x,y) \in \mathcal L_P$ is seven, and the minimum distance is
\begin{equation*}
    d(C_P)= (5-1)^2 - 7 = 9. 
\end{equation*}
In other words, we are sure to be able to correct errors on 
$\lfloor (9-1)/2 \rfloor = 4$ or fewer symbols.

In general, an ideal code will have $d$ and $k$ large with respect to $N$. This means that the code will correct as many errors as possible and convey a lot of information while maintaining short codewords, which are easier to work with computationally. 

\begin{Remark} One can perform certain operations on a polytope without changing the parameters of the code it yields.  If  $t: \bb{Z}^n \to \bb{Z}^n$  is a unimodular affine transformation, that is $t(\mathbf{x}) = M \mathbf{x} + \lambda$, where $M \in \text{GL}(n, \bb{Z})$ is a unimodular matrix and $\lambda$ has integer entries, such that $t$ maps one polytope $P_1 \subseteq \bb{R}^n$ to another $P_2 \subseteq \bb{R}^n$, then we say that $P_1$ and $P_2$ are lattice equivalent.  In this case, the codes $C_{P_1}$ and $C_{P_2}$ are monomially equivalent \cite{LSchwarz}*{Theorem 4}.  Note that two lattice equivalent polytopes have the same number of lattice points, and monomially equivalent codes share the same parameters. \end{Remark}

The motivation for this paper comes from the work of Soprunov and Soprunova \cite{SS2} who construct classes of examples of higher dimensional toric codes and propose the problem of finding an infinite family of good toric codes. To analyze these families of codes, we need to slightly modify our parameters.

\begin{Definition}
The \newword{information rate}, $R(C_P)$, measures how large the dimension of the code is relative to the block length of the code:
\begin{equation*}
    R(C_P) = \frac{k}{N}.
\end{equation*}
\end{Definition}

\begin{Definition}
The \newword{relative minimum distance}, $\delta(C_P)$, measures how large the minimum distance is relative to the block length of the code:
\begin{equation*}
    \delta(C_P) = \frac{d(C_P)}{N}.
\end{equation*}
In the remainder of the paper we will use the simpler notation of $R(P)$ and $\delta(P)$ instead of 
$R(C_P)$ and $\delta(C_P)$. 
\end{Definition}

Returning to Example \ref{ex:UnitBox} again, the information rate is $\displaystyle R(P) = \frac{4}{(q-1)^2}$ and the relative minimum distance is $\displaystyle \delta(P) = \frac{(q-1)^2 - (2q-3)}{(q-1)^2}$. When $q=5$, $R(P) = \frac{1}{4}$ and $\delta(P) = \frac{9}{16}
$.

Given a polytope $P \subset \mathbb{R}^n$, we can view $P$ as sitting in a larger ambient space $\mathbb{R}^m$, for $m>n$.   The next proposition shows that $\delta(P)$ is invariant under the dimension of the embedded space.

\begin{Proposition}
\label{delta dimension invariant}
Let $P \subseteq [0,q-2]^n \subset \bb R^n$ be an integral  convex lattice polytope and
let $m > n$. Let $P'$ be lattice equivalent to the natural embedding of $P$ in $\bb R^m$,
$P \times \{0\}^{m-n}$.
Then working over $\bb F_q$,
$$\delta(P) = \delta(P').$$
\end{Proposition}
\begin{proof}
Since lattice equivalent polytopes have the same minimum
distances \cite{LSchwarz}*{Theorem 4} it suffices to let 
$P' = P \times \{0\}^{m-n}$. We will show that 
$\delta(P') = \delta(P)$ by computing the minimum distance
of $C_{P'}$ and $C_P$ directly.

Let $Z_k(f)$ denote the zero set of the polynomial
$f$ over $(\bb F_q^\times)^k$, where $f$ has $k$ variables. Notice that 
if $f \in \bb F_q[x_1, \ldots, x_k]$ then 
$$\abs{Z_{k+1}(f)} = (q-1)\abs{Z_k(f)}$$
This is because for each $(\bb F_q^\times)^k$ zero of $Z_k(f)$, we get
$(q-1)$ $(\bb F_q^\times)^{k+1}$ zeros of $Z_{k+1}(f)$: one for
each value we set to the extra variable. Thus by induction we have 
for $f \in \bb F_q[x_1, \ldots, x_n]$,
$$\abs{Z_m(f)} = (q-1)^{m-n} \abs{Z_n(f)}$$

Now, since each $f \in \mathcal L_{P'}$ exclusively uses the first 
$n$ of its $m$ available variables by construction, we see that 
each $f \in \mathcal L_{P'}$ is a member of $\bb F_q[x_1, \ldots, x_n].$
Furthermore, since $P'$ is a copy of $P$ in a higher dimension, we 
have that $\mathcal L_P = \mathcal L_{P'}.$
Lastly, note that by definition
$$\delta(P) = \frac{(q-1)^n - \max\limits_{0 \neq f \in \mathcal L_{P}}\abs{Z_n(f)}}{(q-1)^n} = 1 - \max_{0 \neq f \in \mathcal L_P}\frac{\abs{Z_n(f)}}{(q-1)^n}$$
and 
$$\delta(P') = \frac{(q-1)^m - \max\limits_{0 \neq f \in \mathcal L_{P'}}\abs{Z_m(f)}}{(q-1)^m}= 1 - \max_{0 \neq f \in \mathcal L_{P'}}\frac{\abs{Z_m(f)}}{(q-1)^m}$$
It follows that 
$$\delta(P') = 1 - \max_{0 \neq f \in \mathcal L_{P'}}\frac{\abs{Z_m(f)}}{(q-1)^m}
= 1 - \max_{0 \neq f \in \mathcal L_P}\frac{(q-1)^{m-n}\abs{Z_n(f)}}{(q-1)^m}
= 1 - \max_{0 \neq f \in \mathcal L_P}\frac{\abs{Z_n(f)}}{(q-1)^n}
= \delta(P).$$\end{proof}

Conversely, the information rate will change based on the embedded dimension.

\begin{Proposition}\label{prop:inforateEmbedd}
Let $P \subseteq [0,q-2]^n \subset \bb R^n$ be an integral  convex lattice polytope  and
let $m > n$. Let $P'$ be lattice equivalent to the natural embedding of $P$ in $\bb R^m$,
$P \times \{0\}^{m-n}$.
Then working over $\bb F_q$,
$$R(P') = R(P) (q-1)^{n-m} = \frac{R(P)}{(q-1)^{m-n}}.$$
\end{Proposition}
\begin{proof}
This follows from the fact that lattice equivalent polytopes 
have the same number of lattice points and
the definition of information rate.
\end{proof}

Proposition \ref{prop:inforateEmbedd} shows that to maximize the information rate
for a given polytope, we should embed it into the smallest 
dimension possible. Moreover, since the relative minimum 
distance does not depend on the embedding, by Proposition
\ref{delta dimension invariant}, then in terms of both of our
relevant parameters, the smaller dimension of the embedding,
the better. 

Finally, we note one further property of $\delta$ that we will need later.

\begin{Proposition}
\label{deltasubset}
Let $P$ and $Q$ be integral convex polytopes with $P \subseteq Q \subseteq [0,q-2]^n$. Then $\delta(P) \geq \delta(Q)$.
\end{Proposition}
\begin{proof}
%Suppose $P$ and $Q$ are embedded in $\bb R^n$ with all non-negative coordinates. 
With respect to %this dimension and 
the field $\bb F_q$, let $d(C_P)$ 
and $d(C_Q)$ denote the  respective minimum distances. By definition, there exists a polynomial 
$f \in \mathcal{L}_P \subseteq \bb F_q[x_1,\ldots, x_n]$ such that 
$$d(C_P) = (q-1)^n - \abs{Z(f)}$$
Now, since $P \subseteq Q$ it follows that $\mathcal{L}_P \subseteq \mathcal{L}_Q$ so that $f \in \mathcal{L}_Q$. Thus by definition, 
$$d(C_Q) = (q-1)^n - \max_{g \in \mathcal{L}_Q}\abs{Z(g)}  \leq (q-1)^n - \abs{Z(f)} = d(C_P).$$
We then compute $\delta$: %for this choice of embedding dimension $n$:
$$\delta(Q) = \frac{d(C_Q)}{(q-1)^n} \leq \frac{d(C_P)}{(q-1)^n} = \delta(P).$$
\end{proof}

Various authors have computed the minimum distance of toric codes arising from special polytopes.  We record formulas for the minimum distance of toric codes defined by simplices and boxes, given in  \cite{LSchwarz}, and a generalization to the Cartesian product, given in \cite{SS2}.

%Write these in terms of delta (divide by (q-1)^n)
\begin{Theorem}\cite{LSchwarz}*{Corollary 2} \label{thm:mindistSimplex}
Let $\ell < q-1$ and $P_\ell(n)$ be an $n$-dimensional simplex with side length $\ell$. Then the minimum distance of the toric code $C_{P_\ell(n)}$ is
\begin{equation*}
    d(C_{P_\ell(n)}) = (q-1)^n - \ell(q-1)^{n-1}.
    %\delta(C_{P_\ell(n)}) = 1- \frac{\ell}{(q-1)}
\end{equation*}
\end{Theorem}

\begin{Theorem}\cite{LSchwarz}*{Theorem 3} \label{thm:mindistBox}
Let $P_{\ell_1, \ldots ,\ell_n}$ be the $\ell_1 \times \ell_2 \times \ldots \times \ell_n$ rectangular polytope $[0,\ell_1] \times \cdots \times [0,\ell_n] \subseteq \bb{R}^n$. Let $\ell_1, \ldots ,\ell_n$ be small enough so that $P_{\ell_1, \ldots ,\ell_n} \subseteq [0,q-2]^n \subseteq \bb{R}^n$. Then the minimum distance of the toric code $C_{P_{\ell_1, \ldots ,\ell_n}}$ is
\begin{equation*}
    d(C_{P_{\ell_1, \ldots ,\ell_n}}) =  \prod_{i=1}^{n} ((q-1)-\ell_i).
\end{equation*}
\end{Theorem}

\begin{Theorem}\cite{SS2}*{Theorem 2.1} \label{thm:mindistCartesianProduct}
Let $P \subseteq [0,q-2]^n$ and $Q \subseteq [0,q-2]^m$  be integral convex polytopes.  Then
\[d(C_{P\times Q}) = d(C_P)d(C_Q).\]
\end{Theorem}

\section{Infinite Families of Toric Codes} \label{infinite families}

We next investigate infinite families of toric codes, following Soprunov and Soprunova \cite{SS2}.   

\begin{Definition}
    Let $\{P_i\}$ be a sequence of non-empty  integral convex
    polytopes such that each $P_i \subseteq [0,q-2]^{n_i}$ and $n_i \to \infty$ as $i \to \infty$. Let $C_{P_i}$ be the toric code associated with $P_i$.
    We say $\{C_{P_i}\}$ is an \newword{infinite family of toric codes}.% if for some $\delta, R \in [0,1]$ we have  $n_i \to \infty$, $\delta(P_i) \to \delta$, and $R(P_i) \to R$ as $i \to \infty$.  \label{infinte family def}
    \end{Definition}

    Since every polytope corresponds to only one toric code, sometimes we will refer to $\{P_i\}$ as the
    infinite family of toric codes.

    \begin{Definition}
        %A \newword{good infinite family of toric codes} is an infinite family of toric codes, $\{P_i\}$, where neither $\delta$ nor $R$ is $0.$

    % alt definition:
     A \newword{good infinite family of toric codes} is an infinite family of toric codes, $\{P_i\}$, such that $\delta(P_i) \to \delta$ and $R(P_i) \to R$ as $n_i \to \infty,$ where $\delta, R \in (0,1).$ In other words, an infinite family of toric codes is good if $\delta(P_i)$ and $R(P_i)$ approach positive constants as the dimension of the polytopes approaches infinity.
 \end{Definition}
 
In \cite{SS2}, Soprunov and Soprunova construct various infinite families of toric codes; none of which are good families, and remark that it would be interesting to find such a good family.  
    
%should we put one or both of the prototypical examples here?
    
\subsection{Polytope Operations} With the intent of constructing a good infinite family of codes, we derive formulas for the relative minimum distance and information rate of codes corresponding to polytopes constructed using the join and direct sum operations.

\subsubsection{The Join}

%Join and min distance of Join

\begin{Definition}
Let $P \subseteq \bb{R}^n$ and $Q \subseteq \bb{R}^m$ be integral convex polytopes. The \newword{join} of $P$ and $Q$, denoted by $P*Q$, is defined as $$P*Q = \mathrm{conv}(\{(p,\mathbf{0}^m,0) \, | \, p \in P\} \cup \{(\mathbf{0}^n,q,1) \, | \, q \in Q\}) \subseteq \bb{R}^{n+m+1}$$
\end{Definition}

\begin{Example}
Let $P =[0,2] \subseteq \bb{R}$ and $Q = [0,3] \subseteq \bb{R}$ be line segments illustrated below:
\begin{figure}[H]
    \centering
    \subfigure[]{\begin{tikzpicture}[scale=1]
    % axes
    \draw [semithick, ->, >=latex] (0,0) -- (3.5,0);
    \draw [semithick, ->, >=latex] (0,0) -- (0,3.5);
    % axes labels
    \node [below left] at (0,-0.12) {$1$};
    \node [below left] at (1,-0.12) {$x$};
    \node [below left] at (2,-0.02) {$x^2$};
    % square
    \draw [semithick, black] (0,0) -- (2,0) -- cycle;
    % lattice points
    \draw [fill] (0,0) circle [radius=0.06];
    \draw [fill] (1,0) circle [radius=0.06];
    \draw [fill] (2,0) circle [radius=0.06];

    \end{tikzpicture}}
    \qquad
    \subfigure[]{\begin{tikzpicture}[scale=1]
    % axes
    \draw [semithick, ->, >=latex] (0,0) -- (3.5,0);
    \draw [semithick, ->, >=latex] (0,0) -- (0,3.5);
    % axes labels
    \node [below left] at (0,-0.12) {$1$};
    \node [below left] at (1,-0.12) {$x$};
    \node [below left] at (2,-0.02) {$x^2$};
    \node [below left] at (3,-0.02) {$x^3$};

    % square
    \draw [semithick, black] (0,0) -- (3,0) -- cycle;
    % lattice points
    \draw [fill] (0,0) circle [radius=0.06];
    \draw [fill] (1,0) circle [radius=0.06];
    \draw [fill] (2,0) circle [radius=0.06];
    \draw [fill] (3,0) circle [radius=0.06];

    \end{tikzpicture}}
   
    \caption{(a) Polytope $P$ \qquad (b) Polytope $Q$}
    \label{fig:line_segments}
\end{figure}
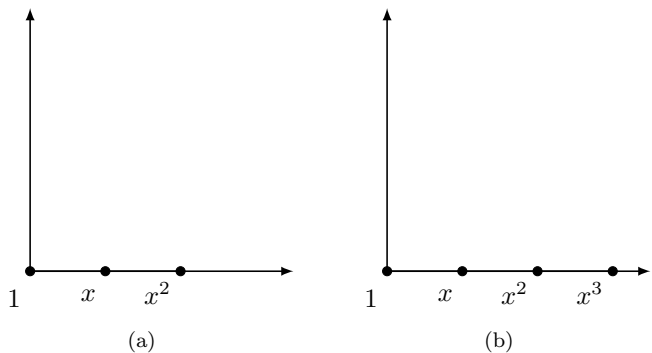
\noindent Then the join of $P$ and $Q$ is the convex hull  of the points $\{(0,0,0), (2,0,0), (0,0,1), (0,3,1)\}$
%the skew lines 
given in Figure \ref{fig:join} below.
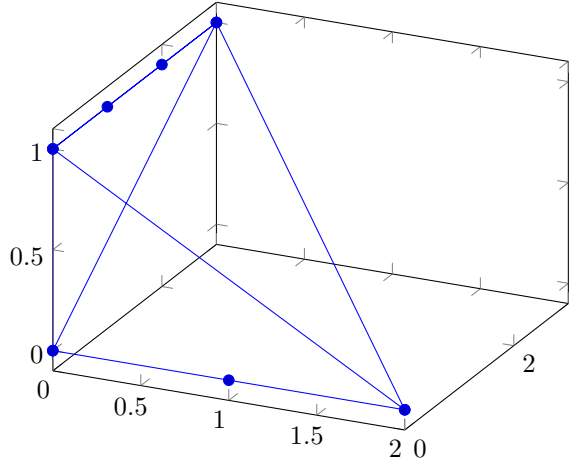
\begin{figure}[H]
\centering
\begin{tikzpicture}
 
\begin{axis}
 
\addplot3 coordinates 
{
    (0,0,0)
    (1,0,0)
    (2,0,0)
    (0,0,1)
    (0,0,0)
    (0,3,1)
    (0,2,1)
    (0,1,1)
    (0,0,1)
    (0,3,1)
    (2,0,0)
};
 
\end{axis}
 
\end{tikzpicture}
\caption{The Polytope $P * Q$}
\label{fig:join}
\end{figure}
In general, the integral lattice points of the join of $P$ and $Q$ are exactly the disjoint union of the integral lattice points of $P$ and the integral lattice points of $Q$. This is because the last coordinate $z$ satisfies $0 \le z \le 1$.  Hence every lattice point in the join has either $z = 0$ or $z = 1$.  When $z = 0$ (resp. $z = 1$), the lattice points are in one-to-one correspondence with the lattice points in $P$ (resp. $Q$).
%This is because the last coordinate (here the ``up'' coordinate) stays strictly between 0 and 1; and agrees with $P$ when $z = 0$ and $Q$ when $z=1$.
\end{Example}

%\begin{Example} Let $P$ be the unit box from Example \ref{ex:UnitBox}, $P= \mathrm{conv}((0,0),(1,0),(0,1),(1,1)) \subseteq \mathbb{R}^2$ and let $Q$ be the unit line segment $Q= \mathrm{conv}(0,1) \subseteq \mathbb{R}$, then the join $P*Q$ is the four-dimensional polytope
%$$ P*Q = \mathrm{conv}((0,0,0,0),(1,0,0,0),(0,1,0,0),(1,1,0,0), (0,0,0,1), (0,0,1,1)).$$
%\end{Example}

\begin{Proposition}
\label{join formula}
Fix a finite field $\bb{F}_q$ and let $P \subseteq [0,q-2]^n \subset \bb{R}^n$ and $Q \subseteq [0,q-2]^m\subset \bb{R}^m$ be integral convex polytopes. Then the maximum number of zeros of $h \in \mathcal{L}_{P*Q}$, is given by
\begin{eqnarray*}
N(P*Q)& = \max&\{(q-1)^{n+m}, N(P)(q-1)^{m+1}, N(Q)(q-1)^{n+1},\\
&& (q-1)N(P)N(Q) + ((q-1)^n-N(P))((q-1)^m-N(Q))\}
\end{eqnarray*}
and the relative minimum distance of $C_{P *Q}$ is
$$\delta(P*Q) = \min\left\{1 - \frac1{q-1}, \delta(P), \delta(Q),
\delta(P) + \delta(Q) - \delta(P)\delta(Q)\frac{q}{q-1}\right\}.$$
\end{Proposition}

\begin{proof}
            Every polynomial $h \in \mathcal{L}_{P*Q}$ has the form $f + x_{n+m+1}g$ for some $f \in \mathcal{L}_P$ and $g \in \mathcal{L}_Q,$ so we can count the number of $\bb{F}_q^\times$-zeros  in $(\bb{F}_q^{\times})^{n+m+1}.$ If both $f$ and $g$ evaluate to zero, then $h$ evaluates to zero at all $(q-1)$ choices for $x_{n+m+1},$ and there are $|Z(f)|$ points in $(\bb{F}_q^{\times})^{n}$ such that $f$ evaluates to zero, and similarly, there are $|Z(g)|$ points in $(\bb{F}_q^{\times})^{m}$ such that $g$ evaluates to zero, so there are $(q-1)|Z(f)||Z(g)|$ points in $(\bb{F}_q^{\times})^{n+m+1}$ such that $h$ evaluates to zero. Additionally, if neither $f$ nor $g$ evaluates to zero, then there is one choice for $x_{n+m+1}$ such that $h$ evaluates to zero, since there is one value for $x_{n+m+1}$ such that $x_{n+m+1}g$ is equal to the additive inverse of $f.$ Moreover, we know that there are $(q-1)^n-|Z(f)|$ points in $(\bb{F}_q^{\times})^{n}$ where $f$ does not evaluate to zero, and similarly there are $(q-1)^m-|Z(g)|$ points in $(\bb{F}_q^{\times})^{m}$ where $g$ does not evaluate to zero, so there are $(1)((q-1)^n-|Z(f)|)((q-1)^m-|Z(g)|)$ additional points in $(\bb{F}_q^{\times})^{n+m+1}$ such that $h$ evaluates to zero. Therefore, we have that for $h = f + x_{n+m+1}g \in \mathcal{L}_{P*Q},$ $$|Z(h)| = (q-1)|Z(f)||Z(g)| + ((q-1)^n - |Z(f)|)((q-1)^m - |Z(g)|))$$ In order to find $N(P*Q),$ we need to determine what values of $|Z(f)|$ and $|Z(g)|$ maximize $|Z(h)|.$ 

Regarding $|Z(f)|$ and $|Z(g)|$ as real variables $x$ and $y$ respectively, and $|Z(h)|$ as a function of these two variables, $F(x,y),$ we use calculus to determine the maximum value of $F(x,y) = (q-1)xy + ((q-1)^n - x)((q-1)^m - y)).$ Taking the partials of this equation, we have $$F_x = (q-1)y - (q-1)^m + y \, \text{ and } \, 
F_y = (q-1)x - (q-1)^n + x. $$
We find that $F_x = 0$ when $y = \frac{(q-1)^m}{q}$ and $F_y = 0$ when $x = \frac{(q-1)^n}{q}.$ %By the second partial derivative test, we know that the critical point $(\frac{(q-1)^n}{q},\frac{(q-1)^m}{q})$ is a saddle point of $F,$ so it does not obtain its maximum value there. This is shown below: 
Computing second partial derivatives, we have:
$$F_{{xx}} = 0, \hspace{.3in}
F_{{yy}} = 0, \hspace{.3in}
F_{{xy}} = q.
$$
Hence $(F_{{xx}})(F_{{yy}}) - (F_{{xy}})^2 = 0 - q^2 < 0$ so that by the second partial derivative test, $(\frac{(q-1)^n}{q},\frac{(q-1)^m}{q})$ is a saddle point of $F.$ Note however that $x \in [0, N(P)]$ and $y \in [0,N(Q)],$ and $F$ is continuous on this set $[0,N(P)] \times [0,N(Q)],$ so $F$ attains its maximum value at some point in the set, and we know that since it is not at a critical point, it occurs on the boundary.

Next, we will show that the maximum value must occur at either the point $(0,0)$ or the point $(N(P),N(Q)).$ Note that $F_x > 0$ when $y > \frac{(q-1)^m}{q}$ and $F_y > 0$ when $x > \frac{(q-1)^n}{q}.$ From this, we have that if $x - \frac{(q-1)^n}{q}$ and $y - \frac{(q-1)^m}{q}$ do not have the same sign, then $F$ will not be at its maximum. Additionally, if $x = N(P)$ and $y \neq N(Q),$ $(\mathrm{and} \,\, y > \frac{(q-1)^m}{q})$ then $F$ is not at its maximum, since $F$ will increase as $y$ increases. In the same way, if $x = 0$ and $y \neq 0,$ $(\mathrm{and} \,\, |Z(g)| < \frac{(q-1)^m}{q})$ then $F$ is not at its maximum, since $F$ will increase as $y$ decreases. A similar analysis holds for when $y = 0$ and $x \neq 0$ and when $y = N(Q)$ and $x \neq N(P).$ Therefore, $F$ obtains its maximum value at either $(0,0)$ or $(N(P),N(Q)),$ so $$N(P*Q) = \mathrm{max}\{(q-1)^n(q-1)^m, (q-1)N(P)N(Q) + ((q-1)^n-N(P))((q-1)^m-N(Q)))\}.$$ Note that when $f$ is identically zero, then $N(P)=(q-1)^n$ so $$(q-1)N(P)N(Q) + ((q-1)^n-N(P))((q-1)^m-N(Q))) = (q-1)(q-1)^nN(Q) + (0)((q-1)^m-N(Q))) = (q-1)^{n+1}N(Q).$$ Similarly, when $g$ is identically zero, then $N(Q) = (q-1)^m$ so $$(q-1)N(P)N(Q) + ((q-1)^n-N(P))((q-1)^m-N(Q))) = (q-1)N(P)(q-1)^m + ((q-1)^n-N(P))(0) = (q-1)^{m+1}N(P).$$ We include these cases for completeness, so we have that \begin{eqnarray*}
N(P*Q)& = \max&\{(q-1)^{n+m}, N(P)(q-1)^{m+1}, N(Q)(q-1)^{n+1},\\
&& (q-1)N(P)N(Q) + ((q-1)^n-N(P))((q-1)^m-N(Q))\}.
\end{eqnarray*} 
Since $\displaystyle \delta(P*Q) = \frac{(q-1)^{n+m+1}-N(P*Q)}{(q-1)^{n+m+1}}$, we have
$$\delta(P*Q) = \min\left\{1 - \frac1{q-1}, \delta(P), \delta(Q),
\delta(P) + \delta(Q) - \delta(P)\delta(Q)\left( \frac{q}{q-1}\right)\right\}.$$
\end{proof}

\begin{Proposition}
\label{join saddlepoint}
Let $P \subseteq [0,q-2]^n$ and $Q \subseteq [0,q-2]^m,$ and let $N(P)$ and $N(Q)$ be as above. If $N(P) \geq 2\frac{(q-1)^n}q$ and $N(Q) \geq 2\frac{(q-1)^m}q$ then
$$(q-1)^{n+m} \leq  (q-1)N(P)N(Q) + ((q-1)^n - N(P))((q-1)^m - N(Q))$$
\end{Proposition}
\begin{proof}
Suppose $N(P) \geq 2\frac{(q-1)^n}q$ and $N(Q) \geq 2\frac{(q-1)^m}q$ and 
let $F : \bb R^2 \to \bb R$ by
$$F(x,y) = (q-1)xy + ((q-1)^n - x)((q-1)^m - y).$$
We wish to show that
$$(q-1)^n(q-1)^m \leq F(N(P),N(Q)).$$
By the analysis in the proof of Proposition \ref{join formula} we know that 
$F(x,y)$ is monotonically increasing with both $x$ and $y$ when both
$x \geq \frac{(q-1)^n}q$ and $y \geq \frac{(q-1)^m}q$ since both partial derivatives are 
non-negative in this region.

In particular, we have since $q \geq 1$
$$N(P) \geq 2 \frac{(q-1)^n}q \geq \frac{(q-1)^n}q$$
and similarly for $N(Q)$. Thus,

\begin{align*}
F(N(P), N(Q)) &\geq F\left(2 \frac{(q-1)^n}q, 2 \frac{(q-1)^m}q\right)\\
&= (q-1)\frac{2(q-1)^n}q\frac{2(q-1)^m}q \\
&\qquad + \left((q-1)^n - \frac{2(q-1)^n}q\right)\left((q-1)^m - \frac{2(q-1)^m}q\right)\\
&= (q-1)^{n+m}\Big[(q-1)(2/q)^2 + (1-2/q)^2\Big]\\
&= (q-1)^{n+m}\Big[4/q - 4/q^2 + 1 - 4/q + 4/q^2\Big]\\
&= (q-1)^{n+m}.
\end{align*}
\end{proof}

\begin{Corollary}
\label{join corollary}
If both $P \subseteq [0,q-2]^n$ and $Q\subseteq [0,q-2]^m$ contain either a lattice segment of length at least $2$ or a square of 
side length $1$, then $(q-1)^{n+m} \leq (q-1)N(P)N(Q) + ((q-1)^n-N(P))((q-1)^m-N(Q))).$ 

%which means that $$1 - \frac1{q-1} \geq \delta(P) + \delta(Q) - \delta(P)\delta(Q)\left(1 - \frac1{q-1}\right),$$ 

Thus, $$\delta(P*Q) = \min\left\{\delta(P), \delta(Q),
\delta(P) + \delta(Q) - \delta(P)\delta(Q)\left( \frac{q}{q-1}\right)\right\}.$$
\end{Corollary}
\begin{proof}
Suppose $P$ is a polytope with a length 2 lattice segment. We can say that the
three points of this segment are without loss of generality (by applying a unimodular affine  
transformation)
$$(0, 0, \ldots, 0)$$
$$(1, 0, \ldots, 0)$$
$$(2, 0, \ldots, 0)$$
Then we see that for any $a,b \in \bb F_q^\times$
$$f = (x_1-a)(x_1-b) = ab - (a+b)x_1 + x_1^2 \in \mathcal{L}_P.$$
We can count the zeros of $f$ as 
$$Z(f) = 2 (q-1)^{n-1}$$
since there are two values for $x_1$ and $(q-1)$ values for each of the $(n-1)$ other coordinates
that result if $f=0$. In particular since $q$ is always a prime power we have that $q \geq 2$ and
$$2\frac{(q-1)^n}q \leq 2\frac{(q-1)^n}{q-1} = 2 (q-1)^{n-1} = Z(f) \leq N(P).$$
Thus any polytope with a length 2 lattice segment meets the conditions for 
Proposition \ref{join saddlepoint}.

Similarly if $P$ is a polytope with a length 1 lattice square then we can apply a unimodular affine transformation so that we have 
$$f = (x_1 - a)(x_2 - b) \in  \mathcal{L}_P, \text{ where $a \neq b$.}$$
We see that 
$$Z(f) = (q-1)^{n-1} + (q-2) (q-1)^{n-2} = (q-1)^{n-2}(q-1 + q-2).$$
Since $q\geq 2$ we have
$$2(q-1)^2 = 2q^2 - 4q + 2 \leq 2q^2 - 3q = q(q-1 + q - 2).$$
In particular, 
$$2\frac{(q-1)^n}q \leq \frac{(q-1)^{n-2}}q q(q-1 + q-2) = (q-1)^{n-2}(q-1 + q-2) = Z(f) \leq N(P).$$
Thus any polytope with a length 1 lattice square meets the conditions for 
Proposition \ref{join saddlepoint} and we have proved the first statement.  

We then note that in our case, Proposition \ref{join formula} gives

$$
N(P*Q)= \max\{ N(P)(q-1)^{m+1}, N(Q)(q-1)^{n+1}, (q-1)N(P)N(Q) + ((q-1)^n-N(P))((q-1)^m-N(Q))\}
$$
and so 
$$\delta(P*Q) = \min\left\{ \delta(P), \delta(Q),
\delta(P) + \delta(Q) - \delta(P)\delta(Q)\left( \frac{q}{q-1}\right)\right\}.$$
\end{proof}

%example sequence of joins

\subsubsection{The Direct Sum}

\begin{Definition}
The \newword{subdirect sum} of two integral convex polytopes $P \subseteq \mathbb{R}^n$ and $Q \subseteq \mathbb{R}^m$ is denoted by $P \oplus Q$ and defined as $$P \oplus Q := \mathrm{conv}(\{(p,\mathbf{0}) \in \mathbb{R}^{n+m} \, | \, p \in P\} \cup  \{(\mathbf{0},q) \in \mathbb{R}^{n+m} \, | \, q \in Q\})$$ If $P$ and $Q$ both contain the origin, then $P \oplus Q$ is called the \newword{direct sum} of $P$ and $Q$.
\end{Definition}

%NEW
Let $P \subseteq [0,q-2]^n$ be an integral convex polytope that contains the origin and let $P' = P \oplus [0, \ell]$, where $\ell \leq q-2$. Our goal is to compute $\delta(P')$.  For $0 \leq i \leq \ell$, let $P_i$ be the greatest integral polytope contained in the scaled polytope $\frac{\ell-i}{\ell} P$; in other words each $P_i$ is a slice of $P'$, where $x_{n+1}$, the $(n+1)^{st}$ coordinate, equals $i$ and that slice is projected to $\bb{R}^n$.  Note that $P_0 = P$, $P_\ell$ is the origin, and $P_i \subseteq P_{i-1}$, so that 
\[d(P) = d(P_0) \leq d(P_1) \leq \dots \leq d(P_\ell)=(q-1)^n. \]
By \cite[Cor. 4.3]{S2015},
\[d(C_{P'}) \geq \min_{0 \leq i \leq \ell}(q-1-i)d(C_{P_i}).\]

\begin{Proposition}
\label{directsum formula}
In the above setup, if the minimum of $(q-1-i)d(C_{P_i})$, as $i$ ranges from $0$ to $\ell$, is achieved at either $i=0$ or $i=\ell$, then
$N(P') = \mathrm{max}\{N(P)(q-1),\ell(q-1)^n\}$ 
and
$$\delta(P') = \min\left\{\delta(P), 1 - \frac{\ell}{q-1}\right\} = \min\left\{\delta(P), \delta([0,\ell])\right\}.$$
\end{Proposition}

\begin{proof}
Suppose first that the minimum occurs when $i=0$.  Then by \cite[Cor. 4.3]{S2015}, $d(C_{P'}) \geq (q-1)d(C_{P_0})$ so that
\[(q-1)^{n+1} - N(P') \geq (q-1)((q-1)^n - N(P)) = (q-1)^{n+1} - (q-1)N(P),\]
hence
\[N(P') \leq (q-1)N(P).\]
Consider a polynomial $f \in \mathcal{L}_{P}$ which has the maximum $N(P)$ zeros, then $x_{n+1}f \in \mathcal{L}_{P'}$ and has $(q-1)N(P)$ zeros, so $N(P') = (q-1)N(P)$.

 If the minimum occurs at $i=\ell$, then we have $d(C_{P'}) \geq (q-1-\ell)d(C_{P_\ell})= (q-1-\ell)(q-1)^n,$ so that
\[(q-1)^{n+1} - N(P') \geq (q-1-\ell)(q-1)^n,\]
and 
\[N(P') \leq \ell(q-1)^n.\]
Consider the polynomial $f \in \mathcal{L}_{P'}$, $f:=(x_{n+1}-a_1)(x_{n+1}-a_2) \cdots (x_{n+1}-a_\ell)$ with $a_k \in \bb{F}_q^\times$ all distinct.  Then $f$ has $\ell(q-1)^n$ zeros, so that $N(P') = \ell(q-1)^n$.  

 Therefore, $$N(P') = \max\{N(P)(q-1),\ell(q-1)^n\}. $$
Thus, $$d(C_{P'}) = \min\{(q-1)^{n+1}  - N(P)(q-1),  (q-1)^{n+1}  -\ell(q-1)^n\} = \mathrm{min}\{(q-1)d(C_{P}),(q-1)^n(q-1-\ell)\}, $$
and
$$ \delta(P') = \min\{\delta(P),\delta([0,\ell])\}.$$
 \end{proof}
 
 \begin{Remark}
 We conjecture that the the minimum of $(q-1-i)d(C_{P_i})$, as $i$ ranges from $0$ to $\ell$, is always achieved at either $i=0$ or $i=\ell$.  In fact, numerous examples suggest that 
 the sequence $\{(q-1-i)d(C_{P_i})\}$ is unimodal: either it increases (with minimum at $i=0$), decreases (with minimum at $i=l$), or it increases until $i=k$, for some $k$, and then decreases, hence achieving a minimum at either $i=0$ or $i=\ell.$  %As we'll see in Example \ref{ex:dirsum}, varying $\ell$, can change the behaviour of $(q-1-i)d(C_{P_i})$.
 \end{Remark}

 We next prove a special case in which the condition of Proposition \ref{directsum formula} holds.
 
 \begin{Proposition}\label{Prop:specialCase}  Let $P \subseteq [0,q-2]^n$ be an integral convex polytope that contains the origin and $P' = P \oplus [0, \ell]$, where $\ell \leq q-2$.  We additionally assume the following:
 \begin{enumerate}
     \item there is a polynomial $f \in \mathcal{L}_P$ with $Z(f) = N(P)$ that is given by a line segment through the origin of lattice length $m$, and
     \item in each $P_i$, there is a polynomial $f_i \in \mathcal{L}_{P_i}$ with $Z(f_i) = N(P_i)$ that is given by that same line segment, but scaled to be in $P_i$.
 \end{enumerate}
 Then $(q-1-i)d(C_{P_i})$, as $i$ ranges from $0$ to $\ell$, achieves a minimum at either $i=0$ or $i=\ell$.
 \end{Proposition}
 
 \begin{proof}
By our first assumption on $P$, we have that 
\[d(C_P) = (q-1)^{n-1}(q-1-m),\]
and by the second assumption on the $P_i$, we have that
\[d(C_{P_i}) = (q-1)^{n-1}(q-1 - \lfloor \frac{\ell-i}{\ell} \cdot m \rfloor). \]
We want to show that $(q-1-i)d(C_{P_i}) = (q-1-i)(q-1)^{n-1}(q-1 - \lfloor \frac{\ell-i}{\ell} \cdot m \rfloor)$ attains a minimum at either $i=0$ or $i=\ell$.  

Consider the function $G(x) = (q-1)^{n-1}(q-1-x)(q-1 - (\frac{\ell-x}{\ell})x)$, then $G'(x) =(q-1)^{n-1}((q-1)(\frac{m}{\ell}-1)+m-2(\frac{m}{\ell})x)$ and when $G'(x)=0$, we have 
\[x = \frac{((q-1)(\frac{m}{\ell}-1)+m) \ell}{2m}.\]
Since $G''(x) =-2$, $\displaystyle x_0:=\frac{((q-1)(\frac{m}{\ell}-1)+m) \ell}{2m}$ is a maximum of $G(x)$.  We next consider three different cases:  when $x_0 \leq 0$, $x_0 > \ell $, and $0 < x_0 < \ell$.

\begin{itemize}
\item If $x_0 \leq 0$, then $G(x)$ is decreasing on $[0,\ell]$, hence $(q-1-i)d(C_{P_i})$ is decreasing on $[0,\ell]$, and attains a minimum at $i = \ell$.

\item If $x_0 > \ell $, then $G(x)$ is increasing on $[0,\ell]$, hence $(q-1-i)d(C_{P_i})$ is increasing on $[0,\ell]$, and attains a minimum at $i = 0$.

\item If $0 < x_0 < \ell$, then $G(x)$ is increasing until $x_0$ and then decreasing, this tells us that $(q-1-i)d(C_{P_i})$ is unimodal, and in particular attains a minimum at either $i=\ell$ or $i=0$.
\end{itemize}
 \end{proof}

\begin{Example}
\label{dirsimpex}
We can use Propositions \ref{directsum formula} and \ref{Prop:specialCase} to find the relative minimum distance of infinite families of toric codes. For example, consider the family of simplices where $P_0 = [0,\ell_0]$ and $$P_{i} = P_{i-1} \oplus [0,\ell_i]$$ where $0 \leq \ell_i \leq q-2$ for all $i.$ The minimum distance of these general simplices was previously computed by Little and Schwarz using Vandermonde matrices \cite{LSchwarz}*{Theorem 2}. By Theorem \ref{thm:mindistSimplex}, the minimum distance of each $P_i \subseteq [0,q-2]^{n_i}$ is given by $(q-1)^{n_i} - \ell(q-1)^{n_i-1},$ where $\ell = \max\limits_{0 \leq j \leq i}\{\ell_j\}.$ Then setting $P_{i+1} = P_i \oplus [0, \ell_{i+1}],$ we have that each cross-section, $P_{i_j},$ satisfies assumption (2) of Proposition \ref{Prop:specialCase}. Thus each $P_i$ satisfies the criteria for $P$ given in Proposition \ref{Prop:specialCase}. Therefore, if $\ell = \max\limits_{0 \leq j \leq i}\{\ell_j\}$ then $$\delta(P_i) = \min \{\delta(P_{i-1}),\delta([0,\ell_i])\} = \min_{0 \leq j \leq i}  \left\{1-\frac{\ell_{j}}{q-1} \right\} = 1 - \frac{\ell}{q-1}.$$
\end{Example}

We next work through a set of concrete examples in which Propositions \ref{directsum formula} and \ref{Prop:specialCase} apply, and we illustrate how the behavior of $(q-1-i)d(C_{P_i})$ can change when the original polytope $P$ is fixed and $\ell$ varies.  

\begin{Example}\label{ex:dirsum}
Consider the polytope $P := \text{conv}( (0,0), (2,3), (4,2)) \subseteq \bb{R}^2$ which has seven lattice points:

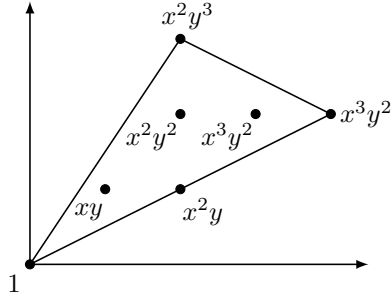
\begin{figure}[h!] \label{fig:p0}
    \begin{center}
    \begin{tikzpicture}[scale=1]
    % axes
    \draw [semithick, ->, >=latex] (0,0) -- (4.5,0);
    \draw [semithick, ->, >=latex] (0,0) -- (0,3.5);
    % axes labels
    \node [below left] at (0,-0.02) {$1$};
    \node [left] at (1.1,0.7) {$xy$};
    \node [below] at (2.3,1) {$x^2y$};
    \node [below left] at (2.1,2) {$x^2y^2$};
    \node [below left] at (3.1,2) {$x^3y^2$};
    \node [above] at (2.1,3) {$x^2y^3$};
    \node [right] at (4.,2) {$x^4y^2$};
    % square
    \draw [semithick, black] (0,0) -- (2,3) -- (4,2) -- (0,0) -- cycle;
    % lattice points
    \draw [fill] (0,0) circle [radius=0.06];
    \draw [fill] (2,3) circle [radius=0.06];
    \draw [fill] (4,2) circle [radius=0.06];
    \draw [fill] (1,1) circle [radius=0.06];
    \draw [fill] (2,1) circle [radius=0.06];
    \draw [fill] (2,2) circle [radius=0.06];
    \draw [fill] (3,2) circle [radius=0.06];
    \end{tikzpicture}
    \end{center}
    \caption{The polytope $P$.} 
    \end{figure}
    Note that $N(P) = 2(q-1),$ and $f = (xy-a)(xy-b) = x^2y^2-(b+a)xy+ab$ for $a \neq b \in \mathbb{F}_q^{\times},$ is a polynomial that has $N(P)$ zeros, corresponding to a line through the origin of lattice length $2$.

We first analyze the direct sum of $P$ and $[0, \ell]$ when $\ell =5$; let $P' = P \oplus [0,5] \subseteq \bb{R}^3$, which has 18 lattice points.  In this case the slices $P_i \subseteq \bb{R}^2$ are:

  \[\begin{tabular}{|c|c|c|} \hline
  $i$ & $P' \cap \{z=i\}$ & $P_i$ \\ \hline
  $0$  & $P$ & $P_0=P$ \\  \hline
  $1$ & $\{(0,0,1), (2,1,1),(1,1,1),(2,2,1)\} $ & $P_1=\text{conv}((0,0), (2,1),(1,1),(2,2))$ \\  \hline
  $2$  & $\{(0,0,2), (2,1,2),(1,1,2)\}$ & $P_2=\text{conv}((0,0), (2,1),(1,1))$ \\  \hline
  $3$  & $\{((0,0,3), (1,1,3)\}$ & $P_3=\text{conv}((0,0), (1,1))$ \\ \hline
  $4$ & $\{(0,0,4)\}$ & $P_4=\{(0,0)\}$ \\ \hline
  $5$ & $\{(0,0,5)\}$ & $P_5=\{(0,0)\}$ \\ \hline
  
  \end{tabular}
  \]
  
 % So, $P_0=P$, $P' \cap \{z=1\} = \{(0,0,1), (2,1,1),(1,1,1),(2,2,1)\}$, so that $P_1=\text{conv}((0,0), (2,1),(1,1),(2,2))$;  $P' \cap\{z=2 \} = \{(0,0,2), (2,1,2),(1,1,2)\}$,  so that $P_2=\text{conv}((0,0), (2,1),(1,1))$;  $P'\cap \{z=3\} = \{((0,0,3), (1,1,3)\}$, so that $P_3 = \text{conv}((0,0), (1,1))$; $P'\cap \{z=4\} = \{(0,0,4)\}$ and $P'\cap \{z=5\} = \{(0,0,5)\}$, so that $P_4 = P_5=\{(0,0)\}$.  

%want all on same line
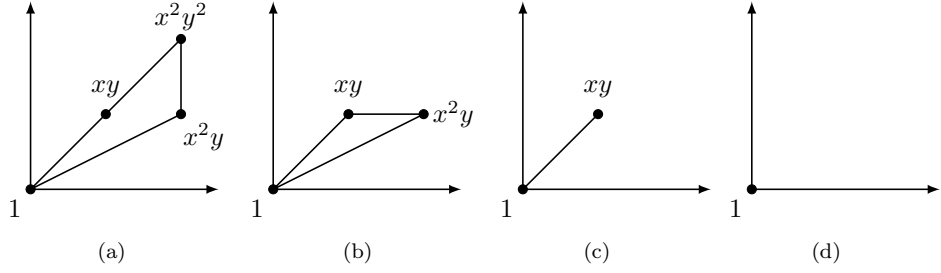
\begin{figure}[H]
    \centering
    \subfigure[]{\begin{tikzpicture}[scale=1]
    % axes
    \draw [semithick, ->, >=latex] (0,0) -- (2.5,0);
    \draw [semithick, ->, >=latex] (0,0) -- (0,2.5);
    % axes labels
    \node [below left] at (0,-0.02) {$1$};
    \node [above] at (1,1.1) {$xy$};
    \node [below] at (2.3,1) {$x^2y$};
    \node [above] at (2,2) {$x^2y^2$};
    % square
    \draw [semithick, black] (0,0) -- (2,2) -- (2,1) -- (0,0) -- cycle;
    % lattice points
    \draw [fill] (0,0) circle [radius=0.06];
    \draw [fill] (1,1) circle [radius=0.06];
    \draw [fill] (2,1) circle [radius=0.06];
    \draw [fill] (2,2) circle [radius=0.06];
    \end{tikzpicture}}
    \subfigure[]{\begin{tikzpicture}[scale=1]
    % axes
    \draw [semithick, ->, >=latex] (0,0) -- (2.5,0);
    \draw [semithick, ->, >=latex] (0,0) -- (0,2.5);
    % axes labels
    \node [below left] at (0,-0.02) {$1$};
    \node [above] at (1,1.1) {$xy$};
    \node [right] at (2,1) {$x^2y$};
    % square
    \draw [semithick, black] (0,0) -- (1,1) -- (2,1) -- (0,0) -- cycle;
    % lattice points
    \draw [fill] (0,0) circle [radius=0.06];
    \draw [fill] (1,1) circle [radius=0.06];
    \draw [fill] (2,1) circle [radius=0.06];
    \end{tikzpicture}}
    \subfigure[]{\begin{tikzpicture}[scale=1]
    % axes
    \draw [semithick, ->, >=latex] (0,0) -- (2.5,0);
    \draw [semithick, ->, >=latex] (0,0) -- (0,2.5);
    % axes labels
    \node [below left] at (0,-0.02) {$1$};
    \node [above] at (1,1.1) {$xy$};
    % square
    \draw [semithick, black] (0,0) -- (1,1) -- cycle;
    % lattice points
    \draw [fill] (0,0) circle [radius=0.06];
    \draw [fill] (1,1) circle [radius=0.06];
    \end{tikzpicture}}
    \subfigure[]{\begin{tikzpicture}[scale=1]
    % axes
    \draw [semithick, ->, >=latex] (0,0) -- (2.5,0);
    \draw [semithick, ->, >=latex] (0,0) -- (0,2.5);
    % axes labels
    \node [below left] at (0,-0.02) {$1$};

    \draw [semithick, black] (0,0);
    % lattice points
    \draw [fill] (0,0) circle [radius=0.06];
    \end{tikzpicture}}
   
    \caption{(a) Polytope $P_1$ (b) Polytope $P_2$ (c) Polytope $P_3$ (d) Polytope $P_4=P_5$}
    \label{fig:Pi}
\end{figure}

% \begin{figure}[h!] \label{fig:p2}
%     \begin{center}

%     \end{center}
%     \caption{The polytope $P_2$.} 
%     \end{figure}    

%     \caption{The polytope $P_3$.} 

%Viewing each $P_i$ as a polytope in the plane we see that $P_5 =P_4\subseteq P_3 \subseteq P_2 \subseteq P_1 \subseteq P_0$, so that
%\[d(C_{P_0}) \leq d(C_{P_1}) \leq d(C_{P_2}) \leq d(C_{P_3}) \leq d(C_{P_4}) = d(C_{P_5}).\]
%By \cite{S2015}*{Cor 4.3}
%\[d(C_{P'}) \geq \min_{0\leq i \leq 5}(q-1-i)d(C_{P_i})\]
We then compute the minimum distances of each $C_{P_i}$:

\[\begin{tabular}{|c|c|c|}  \hline
$i$ & $d(C_{P_i})$ & $(q-1-i)d(C_{P_i})$ \\ \hline
$0$& $(q-1)^2-2(q-1)$ & $(q-1)^2(q-3)$ \\ \hline
$1$ & $(q-1)^2-2(q-1)$ & $(q-1)(q-2)(q-3)$ \\ \hline
$2$ & $(q-1)^2-(q-1)$ & $(q-1)(q-2)(q-3)$ \\ \hline
$3$& $(q-1)^2-(q-1)$ & $(q-1)(q-2)(q-4)$ \\ \hline
$4$& $(q-1)^2$ & $(q-1)^2(q-5)$ \\ \hline
$5$& $(q-1)^2$ & $(q-1)^2(q-6)$ \\ \hline
\end{tabular}\]

Here we see that $(q-1-i)d(C_{P_i})$ is decreasing, so the minimum occurs when $i=5$, and by Proposition \ref{directsum formula}
\[d(C_{P'}) = (q-1)^3 - 5(q-1)^2.\]

We next perform a similar analysis, keeping $P$ the same but varying $\ell$.  Let $Q':= P \oplus[0,3] \subset \bb{R}^3$, and for $0 \leq i \leq 3$ denote the slices by $Q_i$.   

\[\begin{tabular}{|c|c|c|c|}  \hline
$i$ & $Q_i$ & $d(C_{Q_i})$ & $(q-1-i)d(C_{Q_i})$ \\ \hline
$0$& $P$ & $(q-1)^2-2(q-1)$ & $(q-1)^2(q-3)$ \\ \hline
$1$ & $\text{conv}((0,0),(2,1),(1,1))$ & $(q-1)(q-2)$ & $(q-1)(q-2)^2$ \\ \hline
$2$ &$\{(0,0)\}$& $(q-1)^2$ & $(q-1)^2(q-3)$ \\ \hline
$3$ &$\{(0,0)\}$& $(q-1)^2$ & $(q-1)^2(q-4)$ \\ \hline

\end{tabular}\]
With this value of $\ell$, we see that $(q-1-i)d(C_{Q_i})$ increases from $i=0$ to $i=1$, and then decreases, attaining a minimum at $i=3$.

%This example supports our conjecture that Proposition \ref{directsum formula} holds for all polytopes $P.$
\end{Example}

\begin{Example}\label{ex:boxes} There are certainly examples that don't satisfy the assumptions of Proposition \ref{Prop:specialCase}.  For example, let $P \subseteq \mathbb{R}^2$ be the square $[0,3] \times [0,3]$, then a polynomial $f \in \mathcal{L}_P$ with maximum zeros is $f=(x-a_1)(x-a_2)(x-a_3)(y-b_1)(y-b_2)(y-b_3)$ for distinct $a_i\in \bb{F}_q^\times$ and $b_i \in \bb{F}_q^\times$, which corresponds to the full box, not just a line segment.  If $P'=P \oplus [0,\ell]$, then each slice $P_i$ will also be a square of the form $[0,i] \times [0,i]$ for $0 \leq i \leq 3$, and hence each polynomial $f_i \in \mathcal{L}_{P_i}$ with maximum number of zeros is not given by a line segment.  However, the conditions of Proposition \ref{directsum formula} are met for any $\ell$; that is, the minimum of $(q-1-i)d(C_{P_i})$, as $i$ ranges from $0$ to $\ell$ is achieved at $i=0$ or $i=\ell$.  Indeed, if $\ell \leq 3$, the sequence $\{(q-1-i)d(C_{P_i})\}$ is increasing for $i=0$ to $i=\ell$; and if $\ell >3$  the sequence $\{(q-1-i)d(C_{P_i})\}$ increases and then decreases.
\end{Example}

%END NEW

\subsection{Examples of infinite families of toric codes} We now give examples of infinite families of toric codes constructed using the polytope operations defined above along with some other common operations. Both the family of boxes and simplices are considered in \cite{SS2}*{Theorem 3.1} and \cite{SS2}*{Proposition 4.1}, but as one general construction. To clarify what fails about each family, we separate them.

\subsubsection{Family of Boxes}
Let $P_0 = [0,\ell_0]$ and $$P_i = P_{i-1} \times [0,\ell_i]$$ where $0 \leq \ell_i \leq q-2$ for all $i.$ 

\begin{Proposition}
The infinite family of toric codes corresponding to this sequence of boxes has either $\delta(P_i) \to 0$ or $R(P_i) \to 0.$
\end{Proposition}

%The proof of this fact is given in Proposition 4.1 of \cite{SS2}.

\begin{proof}
By Theorem \ref{thm:mindistBox}, the infinite family of toric codes corresponding to this sequence of polytopes has 
$$\delta(P_i) = \delta(P_{i-1}) \left(1 - \frac{\ell_i}{q-1}\right)$$
and by counting lattice points
$$R(P_i) = R(P_{i-1}) \frac{\ell_i + 1}{q-1}.$$ 
We will show that either $\delta(P_i) \to 0$ or 
$R(P_i) \to 0$. First, assume that $\delta(P_i) \not \to 0$.
This implies that all but finitely many $\ell_i = 0$
(if infinitely many $\ell_i \ne 0$ then $\delta(P_i)$ is comparable to a geometric sequence of rate $(q-2)/(q-1) < 1$ and thus converges to 0).  But with this restriction on $\ell_i,$ $R(P_i) \to 0$ as $i \to \infty$ (since now $R(P_i)$ is comparable to a geometric sequence of rate $1/(q-1)$).
Similarly, assume $R(P_i) \not \to 0 .$ This implies that all but finitely many $\ell_i = q-2,$ but then $\delta \to 0$ as $i \to \infty.$ Therefore, there is no way that neither $\delta(P_i)$ nor $R(P_i)$ tends to $0$ as $i \to \infty,$ so this infinite family of toric codes is not a good infinite family. 
\end{proof}

\subsubsection{Family of Simplices}
Consider the family of simplices where $P_0 = [0,\ell_0]$ and $$P_{i} = P_{i-1} \oplus [0,\ell_i]$$ where $0 \leq \ell_i \leq q-2$ for all $i.$ 

\begin{Proposition}
The infinite family of toric codes corresponding to this sequence of simplices has $\delta(P_i) \geq \delta > 0$ for some $0 < \delta \leq 1.$
\end{Proposition}
\begin{proof}
From Example \ref{dirsimpex}, we have that $$\delta(P_i) = \min \{\delta(P_{i-1}),\delta([0,\ell_i])\} = \min_{j \leq i} \left\{1-\frac{\ell_{j}}{q-1} \right\}.$$ 
Since $\ell_i \leq q-2$ for all $i$, we know that $$\delta(P_i) \geq 1 - \frac{q-2}{q-1} = \frac{1}{q-1} > 0.$$ 
In particular, for any construction of $\{P_i\},$
if $\delta(P_i)$ converges, it will converge to some strictly positive value.
\end{proof}

\begin{Proposition}
The infinite family of toric codes corresponding to this sequence of simplices has $R(P_i) \to 0.$
\end{Proposition}
\begin{proof}
We know that $$R(P_i) \leq \frac1{(q-1)^i}{i + \ell \choose \ell} \leq \frac1{(q-1)^i} \frac{(i+\ell)^\ell}{\ell !}$$ where $\ell = \max {\ell_i}.$  Note $$ \lim_{i \to \infty}  \frac1{(q-1)^i} \frac{(i+\ell)^\ell}{\ell !} \to 0$$ since its numerator is a polynomial in $i$ and its denominator is an exponential in $i.$  Thus $R(P_i)\to 0$ as well.
\end{proof}
Since $R(P_i) \to 0$ for any construction of $\{P_i\},$ this infinite family of simplices also fails to be a good infinite family of toric codes.

\subsubsection{Iteratively Taking the Join of a Polytope with Itself}

Let $P \subseteq[0,q-2]^n$ be an integral convex polytope which contains a length two lattice segment or a unit square, and let $P_0 := P.$ For $k\geq 0$, define $P_{k+1} := P_k * P_k.$ For brevity let 
$\delta_k = \delta(P_k).$
By Corollary \ref{join corollary}, $$\delta_{k+1} = \min\left\{\delta_k,
2\delta_k - \delta_k^2\left(  \frac{q}{q-1}\right)\right\}.$$ 

\begin{Proposition}
For the sequence given above and for $k \geq 2$,
$$\delta_k = \delta_1.$$ 
\end{Proposition}

\begin{proof}
Consider the function of a real variable for $q > 2$
$$f(x) = 2x - x^2\frac{q}{q-1}.$$
Using calculus, we can find the maximum value of this function by considering
the first and second derivatives:
$$f'(x) = 2 - 2x\frac{q}{q-1}, \qquad f''(x) = -2\frac{q}{q-1}.$$
We see that $f'(x) = 0$ exactly when $x = \frac{q-1}{q},$ and since 
$q \geq 1$, we have $f''(x) < 0$ for all $x \in \bb R$. Thus, by the second derivative
test, $f$ attains a local maximum at $x = \frac{q-1}{q},$ whose value is 
$$f\left(\frac{q-1}{q}\right) = 2\frac{q-1}q - \left( \frac{q-1}q\right)^2 \frac{q}{q-1} = \frac{q-1}{q}.$$
We see that this is in fact the global maximum of $f$ since $f$ is a parabola.

Further, we see that $f(x) \geq 0$ whenever
$0 \leq x \leq 1.$ This follows from the fact that $f$ is concave down with
$f(0) = f\left(2\frac{q-1}q\right) = 0$
and $0 \leq 1 \leq 2 \frac{q-1}{q}$.

Finally, notice for $0 \leq x \leq \frac{q-1}{q}$ we have $x\frac{q}{q-1} \leq 1$ and so 
$$f(x) - x = x - x^2 \frac{q}{q-1} = x \left(1 - x\frac{q}{q-1}\right) \geq x(1-1) = 0.$$
That is, whenever $0 \leq x \leq \frac{q-1}{q}$ we 
have, $f(x) \geq x$. 

Thus our sequence, in terms of the function $f$, is given
by
$$\delta_{k+1} = \min\{\delta_k, f(\delta_k)\}.$$
Recall that by the properties of $\delta$ we have 
$0 \leq \delta_0 \leq 1$.
Thus, since $0 \leq f(\delta_0) \leq \frac{q-1}{q}$,
we have
$$0 \leq \delta_1 \leq \min\{\delta_0, \frac{q-1}{q}\}
\leq \frac{q-1}{q}.$$
It follows from the above argument that 
$f(\delta_1) \geq \delta_1$ so that 
$$0 \leq \delta_2 = \delta_1 \leq \frac{q-1}{q}.$$
Continuing by induction we have for all $k \geq 2$
$$\delta_k = \delta_1.$$
\end{proof}

Similar to the previous infinite family of codes, we see that by appropriately setting $\delta_0$, we can get $\delta_k \to \delta > 0.$ But 
again, recall that the number of lattice points of the join of two polytopes is the sum of the number of lattice points of those two polytopes. Then for this family of codes, we have that $$R(P_{k+1}) = \frac{2(\dim(C_{P_{k}}))}{(q-1)^{2n_k + 1}} \leq \frac{2(q-1)^{n_k}}{(q-1)^{2n_k + 1}} = \frac{2}{(q-1)^{n_k+1}} \to 0.$$ Therefore, it fails to be a good family since $R(P_k) = 0.$

%Prototypical Ex and Join Ex?
\section{No Good Family} \label{no good family}
\subsection{Unit Hypercubes}
Given the examples above, we have a strong reason to believe that there is \emph{no 
good infinite family} of toric codes as Soprunov and Soprunova describe. 
In order to
formalize this reasoning, we introduce the following definition.

\begin{Definition}
Let $P$ be an integral convex polytope. We define $M(P)$ to be the largest 
integer $m$ such that an $m$-dimensional unit hypercube is a subset of $P$, up to 
a unimodular affine
transformation. That is 
\begin{equation*}
    M(P) = \max\{ ~ m ~ | ~ \exists \text{ a unimodular affine transformation, } A,  \text{ such that } A([0,1]^{m}) \subseteq P\}.
\end{equation*}
%where $A$ is an unimodular affine transformation. 
\end{Definition}

\begin{Proposition}
\label{boxes->delta0}
Let $\{P_i\}$ be an infinite family of toric codes, and suppose the sequence 
$\{M(P_i)\}$ is unbounded. Then if $\delta(P_i)$ converges, it converges to zero.
\end{Proposition}
\begin{proof}
The proof of this follows from Proposition \ref{deltasubset} and the value of $\delta([0,1]^k).$ 
As usual, let us work over $\bb F_q.$

Let $\epsilon > 0$ be arbitrary and consider a subsequence $\{P_{i_j}\}$ such that 
$M(P_{i_j}) \to \infty$ as $j \to \infty$. Since $M(P_{i_j}) \to \infty$, there exists an $N$ such that whenever $j \geq N$
$$\left(\frac{q-2}{q-1}\right)^{M(P_{i_j})} < \epsilon$$ (this exists because $0 \leq \frac{q-2}{q-1} < 1$).
By the definition of $M$, there exists a unimodular affine transformation
$A$ such that $A([0,1]^{M(P_{i_j})}) \subseteq P_{i_j}$. By Theorem \ref{thm:mindistBox}, we have that for all $j \geq N$
$$\delta([0,1]^{M(P_{i_j})}) = \left(\frac{q-2}{q-1}\right)^{M(P_{i_j})}$$
Since unimodular affine transformations preserve $\delta$, we have
$$0 \leq \delta(P_{i_j}) \leq \delta(A[0,1]^{M(P_{i_j})}) = \delta([0,1]^{M(P_{i_j})}) = \left(\frac{q-2}{q-1}\right)^{M(P_{i_j})} < \epsilon.$$
That is $\delta(P_{i_j}) \to 0$ as $j \to \infty$. Since this subsequence $\delta(P_{i_j})$ converges to zero, and  $\delta(P_i)$ converges by assumption, we have that $\delta(P_i)$ converges to zero.
%Finally, since $\delta(P_i)$ converges by assumption, we have that $\delta(P_i) \to 0$ as $i \to \infty$.
\end{proof}
Note that this means that any family with unbounded $M(P_i)$ cannot be good, since $\delta(P_i)$ cannot converge to a strictly positive value. 

Thus, we have a simple characterization of some families whose $\delta=0$. We conjecture the following:

\begin{Conjecture}\label{conjecture2}
Let $\{P_i\}$ be an infinite family of toric codes, and suppose the sequence 
$\{M(P_i)\}$ is bounded. Then $R(P_i) \to 0$ as $i \to \infty$.
\end{Conjecture}
The 
intuition behind this conjecture is as follows: suppose for some infinite family $\{P_i\}$, $\{M(P_i)\}$ is bounded by $M$.
By definition this means that no $P_i$ contains a lattice equivalent copy of $[0,1]^{M+1}$. Intuitively, this means that each $P_i$ ``does not take
up any $(M+1)$-dimensional volume.'' We might expect that since we know each $P_i \subseteq [0,q-2]^{n_i}$ that $R(P_i) \sim (q-1)^M$. However, we could not find a way to verify this intuition.

Instead, we'll relate our above characterization to the Minkowski length.

\subsection{Minkowski Length}
Recall that if $P$ and $Q$ are integral convex polytopes, we can construct a new integral convex polytope using the Minkowski sum.  

\begin{Definition}
Let $P$ and $Q$ be convex polytopes in $\bb R^n.$ Their \emph{Minkowski sum} is 
$$P + Q = \{ p + q \in \bb R^n \,| \,p \in P, q \in Q\}$$
which is again a convex polytope.
\end{Definition}
Figure \ref{fig:minkowski sum} shows the Minkowski sum of two unit line segments in $\bb{R}^2$:

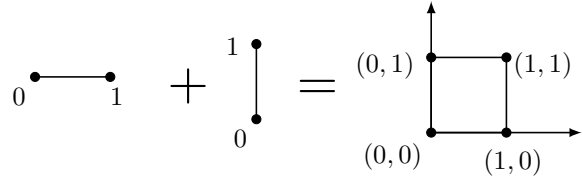
\begin{figure}[h!] 
    \begin{center}
    $$
    \adjustbox{valign=c}{\begin{tikzpicture}[scale=1]
    % axes
    % \draw [semithick, ->, >=latex] (0,0) -- (2,0);
    % \draw [semithick, ->, >=latex] (0,0) -- (0,1.75);
    % axes labels
    \node [below left] at (0,-0.02) {$0$};
    \node [below left] at (1.3,-0.02) {$1$};
    % boundary
    \draw [semithick, black] (0,0) -- (1,0) -- (0,0) -- cycle;
    % lattice points
    \draw [fill] (0,0) circle [radius=0.06];
    \draw [fill] (1,0) circle [radius=0.06];
    \end{tikzpicture}}
    ~\quad\scalebox{2}{+}\quad~
    \adjustbox{valign=c}{\begin{tikzpicture}[scale=1]
    % axes
    %\draw [semithick, ->, >=latex] (0,0) -- (2,0);
    %\draw [semithick, ->, >=latex] (0,0) -- (0,1.75);
    % axes labels
    \node [below left] at (0,-0.02) {$0$};
    \node [below left] at (-0.1,1.2) {$1$};
    % boundary
    \draw [semithick, black] (0,0) -- (0,1) -- (0,0) -- cycle;
    % lattice points
    \draw [fill] (0,0) circle [radius=0.06];
    \draw [fill] (0,1) circle [radius=0.06];
    \end{tikzpicture}}
    ~\quad\scalebox{2}{=}\quad~
    \adjustbox{valign=c}{\begin{tikzpicture}[scale=1]
    % axes
    \draw [semithick, ->, >=latex] (0,0) -- (2,0);
    \draw [semithick, ->, >=latex] (0,0) -- (0,1.75);
    % axes labels
    \node [below left] at (0,-0.02) {$(0,0)$};
    \node [below left] at (-0.1,1.2) {$(0,1)$};
    \node [below left] at (1.6,-0.1) {$(1,0)$};
    \node [below left] at (2,1.2) {$(1,1)$};
    % boundary
    \draw [semithick, black] (0,0) -- (0,1) -- (1,1) -- (1,0) -- (0,0) -- cycle;
    % lattice points
    \draw [fill] (0,0) circle [radius=0.06];
    \draw [fill] (0,1) circle [radius=0.06];
    \draw [fill] (1,1) circle [radius=0.06];
    \draw [fill] (1,0) circle [radius=0.06];
    \end{tikzpicture}}$$
    \end{center}
    \caption{The Minkowski Sum of Lattice Segments.} 
    \label{fig:minkowski sum}
\end{figure}

Let $P$ be an integral convex lattice polytope in $\bb R^n$. Consider a Minkowski decomposition
$$P = P_1 + \ldots + P_\ell$$
into lattice polytopes $P_i$ of positive dimension (that is, none of the $P_i$ are 
single points). Clearly, there are only finitely many such decompositions. We 
let $\ell(P)$ be the largest number of summands in such decompositions of $P$, and
call it the \emph{Minkowski length} of $P$.

\begin{Definition}\cite{SS1}*{Def. 1.1}
The \emph{full Minkowski length} of $P$ is the maximum of the Minkowski lengths of all 
subpolytopes $Q$ in $P$, 
    $$L(P) := \max\{\ell(Q) ~ | ~ Q \subseteq P\}.$$
\end{Definition}

\begin{Example}
Figure \ref{fig:minkowski length} illustrates the Minkowski lengths of various polytopes.

\begin{figure}[h!] 
    \begin{center}
    
    \subfigure[]{\begin{tikzpicture}[scale=1]
    % boundary
    \draw [semithick, black, fill=gray!30] (0,0) -- (2,1) -- (1,2) -- (0,0) -- cycle;
     \draw [semithick, black] (1,1) -- (2,1);
    % axes
    \draw [semithick, ->, >=latex] (0,0) -- (3,0);
    \draw [semithick, ->, >=latex] (0,0) -- (0,2.75);
    % axes labels
    \node [below left] at (0,-0) {$1$};
    \node [below left] at (1,0.9) {$xy$};
    \node [below ] at (2,0.9) {$x^2y$};
    \node [ left] at (.9,1.9) {$xy^2$};

    % lattice points
    \draw [fill] (0,0) circle [radius=0.06];
    \draw [fill] (1,1) circle [radius=0.06];
    \draw [fill] (2,1) circle [radius=0.06];
    \draw [fill] (1,2) circle [radius=0.06];
    
    \end{tikzpicture}}~
    \subfigure[]{\begin{tikzpicture}[scale=1]
    % boundary
    \draw [semithick, black,fill=gray!30] (0,0) -- (0,1) -- (0,2) -- (1,1) -- (2,0) -- (1,0) -- (0,0) -- cycle;
     \draw [semithick, black] (0,0) -- (0,1) -- (1,1) -- (1,0) -- (0,0) -- cycle;
    
    % axes
    \draw [semithick, ->, >=latex] (0,0) -- (3,0);
    \draw [semithick, ->, >=latex] (0,0) -- (0,2.75);
    % axes labels
    \node [below left] at (0,-.05) {$1$};
    \node [below left] at (1,0.9) {$xy$};
    \node [below left] at (0,0.9) {$y$};
    \node [below left] at (1,-0.13) {$x$};
    \node [below left] at (0,1.9) {$y^2$};
    \node [below left] at (2,0) {$x^2$};

    % lattice points
    \draw [fill] (0,0) circle [radius=0.06];
    \draw [fill] (1,0) circle [radius=0.06];
    \draw [fill] (2,0) circle [radius=0.06];
    \draw [fill] (1,1) circle [radius=0.06];
    \draw [fill] (0,1) circle [radius=0.06];
    \draw [fill] (0,2) circle [radius=0.06];

    \end{tikzpicture}}~
    \subfigure[]{\begin{tikzpicture}[scale=1]
    % boundary
    \draw [semithick, black, fill=gray!30] (1,0) -- (0,1) -- (3,3) -- (3,2) -- (2,0) -- (1,0) -- cycle;
    \draw [semithick, black] (1,0) -- (1,1) -- (2,2) -- (3,3) -- (3,2) -- (2,1) -- (1,0) -- cycle;
    % axes
    \draw [semithick, ->, >=latex] (0,0) -- (4,0);
    \draw [semithick, ->, >=latex] (0,0) -- (0,3.75);
    % axes labels
    
    \node [below left] at (1,0.9) {$xy$};
    \node [below left] at (0,0.9) {$y$};
    \node [below left] at (1,-0.1) {$x$};
    \node [below left] at (2,-0.1) {$x^2$};
    \node [below] at (1.9,0.9) {$x^2y$};
    \node [below] at (2,1.9) {$x^2y^2$};
    \node [below right] at (3,1.9) {$x^3y^2$};
    \node [below right] at (3,2.9) {$x^3y^3$};

    % lattice points
    \draw [fill] (1,0) circle [radius=0.06];
    \draw [fill] (0,1) circle [radius=0.06];
    \draw [fill] (1,1) circle [radius=0.06];
    \draw [fill] (2,0) circle [radius=0.06];
    \draw [fill] (2,1) circle [radius=0.06];
    \draw [fill] (2,2) circle [radius=0.06];
    \draw [fill] (3,2) circle [radius=0.06];
    \draw [fill] (3,3) circle [radius=0.06];
    \end{tikzpicture}}~
    \end{center}
    \caption{Polytopes with full Minkowski Lengths (a) 1 (b) 2 (c) 3. The interior polytope represents a subpolytope of the largest Minkowski sum.}
    \label{fig:minkowski length}
\end{figure}
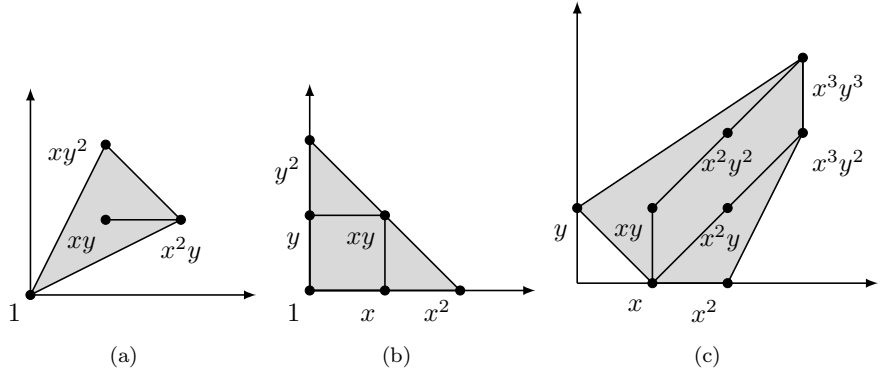
\end{Example}

 \newpage
\begin{Proposition}
\label{minkowskiIffBox}
Let $\{P_i\}$ be an infinite family of toric codes.  Then $\{M(P_i)\}$ is unbounded if and only if $\{L(P_i)\}$ is unbounded.
\end{Proposition}
\begin{proof}
We first assume that $\{M(P_i)\}$ is unbounded and let $N \in \bb N$ be arbitrary. 
Since $\{M(P_i)\}$ is unbounded, there 
exists an $i$ such that $M(P_i) > N$. In other words, there is a lattice equivalent copy of
$[0,1]^{M(P_i)}$ as a subpolytope of $P_i$. We know that the full Minkowski length
is invariant under lattice equivalence \cite{SS1}*{Prop 1.2} and thus, since
$$ [0,e_{1}] + \ldots + [0, e_{M(P_i)}] = [0,1]^{M(P_i)} \subseteq P_i,$$
where the $e_j$ are primitive integer lattice vectors, we have 
$$L(P_i) \geq \ell([0,1]^{M_i}) \geq M(P_i) > N.$$
Since $N$ was arbitrary, we see that $\{L(P_i)\}$ is unbounded.

For the reverse implication, suppose $\{L(P_i)\}$ is unbounded, and let $N \in \bb N$ be arbitrary.
Since $\{L(P_i)\}$ is unbounded there exists an $i$ such that $L(P_i) > (q-2) N$. 
That is $P_i$ has a subpolytope $Q$ that can be written as
$$Q = Q_1 + \ldots + Q_{(q-2)N}$$
where the dimension  of each $Q_\ell$ is at least one. Note that this means each $Q_\ell$ has
at least two points in it, and thus has a lattice equivalent copy of $[0,1]$ as
a subpolytope. 

For each $Q_\ell$, let $v_\ell$ be a direction in which $Q_\ell$ has span. 
That is, two distinct points $p,q \in Q_\ell$ differ by exactly $v_\ell$; or $p + v_\ell = q$.  

Consider the multi-set $V$ given by the sum over $\{v_\ell\}$ when considered as multi-sets:

$$V = \sum_{\ell = 1}^{N(q-2)} \{v_\ell\} / \sim,$$
where $v \sim w$ if there exists a scaling $\lambda \in \bb R$ such that $\lambda v = w$.
We see that since we are summing over singletons the cardinality of this multiset is exactly
$\abs{V} = N(q-2).$

Further, notice that 
if $q-1$ of the $Q_\ell$ share the same $v_\ell$ (or a scaling thereof), then their 
Minkowski sum $Q$ is \emph{wider} than $q-2$ in one of the component directions of the shared $v$. 
However,
this is explicitly forbidden since each $P_j$ lives in $[0,q-2]^{n_j}.$
That is, the multiplicity of $v/ \sim$ in $V$, $m(v)$, is at most $(q-2)$.

Putting these together, we see that 
$$N(q-2) = \abs{V} = \sum_{v \in Supp(V)} m(v) \leq \sum_{v \in Supp(V)} (q-2) = (q-2) \abs{Supp(V)}$$
That is 
$$N \leq \abs{Supp(V)}.$$

In other words, at least $N$ distinct direction vectors appear in $V$. Thus, since the $\{Q_\ell\}$ mutually contain $N$ different direction
vectors, their Minkowski sum will contain a lattice equivalent copy of 
$[0,1]^N$. That is $M(P_i) \geq N$, which completes the proof.

\end{proof}

% $\{L_i\}$ unbounded $\implies \delta = 0$
\begin{Corollary}
Let $\{P_i\}$ be an infinite family of toric codes. Then if $\{L(P_i)\}$ is unbounded and $\delta(P_i)$ converges, 
$\delta(P_i) \to 0$ as $i \to \infty$.
\end{Corollary}
\begin{proof}
This follows directly from Propositions \ref{boxes->delta0} and \ref{minkowskiIffBox}.
\end{proof}

Now we can reformulate  Conjecture \ref{conjecture2} in terms of the Minkowski length:
\begin{Conjecture} \label{conjecture}
Let $\{P_i\}$ be an infinite family of toric codes, and suppose the sequence 
$\{L(P_i)\}$ is bounded. Then $R(P_i) \to 0$ as $i \to \infty$.
\end{Conjecture}

We next examine some special cases under which the conjecture holds:

%  $\{L_i\}$ bounded by $q-3 \implies R = 0$
\begin{Proposition} \label{specialcase1}
Let $\{P_i\}$ be an infinite family of toric codes over $\bb F_q$. If a tail of $\{L(P_i)\}$ 
is bounded above by $(q-3)$ then $R(P_i) \to 0$ as $i \to \infty$.
\end{Proposition}
\begin{proof}
This result follows from \cite{meyer2021number} which states a simple bound for the number of lattice points of $P \subseteq \bb R^n$ with full Minkowski length $L$ as
$$\abs{P \cap \bb Z^n} \leq (L+1)^n.$$
Since a tail of $\{L(P_i)\}$ is bounded by $(q-3)$ there exists an integer
$N$,  such that whenever $i \geq N$, $L(P_i) \leq q-3.$ We have for all $i \geq N$
$$R(P_i) = \frac{\abs{P_i \cap \bb Z^{n_i}}}{(q-1)^{n_i}} \leq \frac{(L(P_i) + 1)^{n_i}}{(q-1)^{n_i}} \leq \left(\frac{q-2}{q-1}\right)^{n_i}.$$
Thus, we see a tail of $\{R(P_i)\}$ converges to 0 and so 
$R(P_i) \to 0$ as $i \to \infty$.
\end{proof}

Note that the proof of Proposition \ref{specialcase1} does not take into account the 
restriction that each $P_i \subseteq [0,q-2]^{n_i}$.
% Each $P_i$ is the Minkowski Sum of unit simplices of any dimension and $\{L_i\}$ is bounded $\implies R = 0$

\begin{Proposition} \label{specialcase2}
Let $\{P_i\}$ be an infinite family of toric codes. Suppose each $P_i$ is 
exactly the Minkowski sum of some number of unit lattice simplices (of 
any dimension not greater than $n_i$).
If $\{L(P_i)\}$ is bounded then $R(P_i) \to 0$ as $i \to \infty$.
\end{Proposition}
\begin{proof}
Given two  integral convex polytopes $P$ and
$Q$ in $\bb R^n$, the Minkowski sum of $P$ and $Q$  has at most $(\abs{P \cap \bb Z^n})(\abs{Q \cap \bb Z^n})$ points.
Since a unit simplex of dimension $m$ has exactly $(m+1)$ points, it follows that $| P_i \cap  \bb{Z}^{n_i}| \leq (n_i + 1)^{L(P_i)}$, given that $P_i$ is 
exactly the Minkowski sum of unit simplices. Let $L$ be such that 
for all $i$, $L(P_i) \leq L$, then 
$$0 \leq R(P_i) \leq \frac{(n_i+1)^{L(P_i)}}{(q-1)^{n_i}} \leq \frac{(n_i+1)^L}{(q-1)^{n_i}}.$$
Since $n_i \to \infty$ as $i \to \infty$ we see that $R(P_i) \to 0$
by the Squeeze Theorem.
\end{proof}

%Where a counterexample would have to live

\section{Future Directions}
We end by posing a few open problems.  
\begin{Problem} Prove or disprove Conjecture \ref{conjecture} (which is equivalent to Conjecture \ref{conjecture2}). \end{Problem}
% Prove our conjecture

If Conjecture \ref{conjecture} is true, then we may consider the following:  define an \emph{$\epsilon$ good code} $C_P$ to be a toric code such that $\delta(P) \geq \epsilon$ and $R(P) \geq \epsilon$. Conjecture \ref{conjecture} claims that as the dimension of $P$ goes to infinity there are no $\epsilon$ good codes for any $\epsilon > 0$. This means that for every $\epsilon > 0$ there exists a largest dimension $N$ for which there exists an $\epsilon$ good code. 

\begin{Problem} For each $\epsilon > 0$, what is the $N$ after which \emph{no} code of dimension $n \geq N$ is $\epsilon$ good? \end{Problem}

The answer to this question also answers ``If I want an optimal code (so that $R + \delta = 1$) where $\delta$ (or $R$) is no worse than $\epsilon$, what is the largest dimension I can use?''

\section*{Acknowledgements}

This research was completed at the \textbf{REU Site: Mathematical Analysis and Applications at the University of Michigan-Dearborn}. We would like to thank the National Science Foundation (DMS-1950102); the National Security Agency (H98230-21); the College of Arts, Sciences, and Letters; and the Department of Mathematics and Statistics
for their support.  We also would like to thank the anonymous referee for their constructive feedback.  

~

\bibliography{REUbib}

\vspace{20pt}

\end{document}